\documentclass[11pt]{article}
= 0.0in \headheight = 0.0in \topmargin = 0.2in
\def\Dj{\hbox{D\kern-.73em\raise.30ex\hbox{-}
\raise-.30ex\hbox{}}}
\def\dj{\hbox{d\kern-.33em\raise.80ex\hbox{-}
\raise-.80ex\hbox{\kern-.40em}}}

\usepackage{amssymb}
\usepackage{color}
\usepackage{amssymb}
\usepackage{float}
\usepackage{amsmath}
\usepackage{caption}
\usepackage{graphicx}
\usepackage{latexsym}
\usepackage{amsthm}
\usepackage{tikz}
\usepackage{lineno}
\usepackage{enumitem}
\usepackage{epsfig}
\usepackage{graphicx}
\usepackage{amsmath,amsthm,amsfonts,amssymb,amscd,cite}
\allowdisplaybreaks
\usepackage{mathrsfs}

\newtheorem{theorem}{Theorem}[section]
\newtheorem{lemma}{Lemma}[section]

\newtheorem{corollary}{Corollary}[section]
\newtheorem{proposition}{Proposition}[section]
\newtheorem{conjecture}[theorem]{Conjecture}

\newtheorem{ques}[theorem]{Question}

\newtheorem{remark}[theorem]{Remark}

\begin{document}

%\baselineskip=0.30in
%
%\vspace*{15mm}

%
%
%\end{center}

\title{Maximal and maximum   induced matchings in  connected graphs}
\author{Bo-Jun Yuan\thanks{ Supported by National Natural Science Foundation of
China(12201559).}, Zhao-Yu Yang, Lu Zheng and Shi-Cai Gong\thanks{Corresponding author.}
\thanks{ Supported by National Natural Science Foundation of
China(12271484).}~\thanks{ E-mail addresses: ybjmath@163.com(B.J. Yuan), yangzhaoyu17@126.com(Z.Y. Yang), zhenglu@zust.edu.cn(L. Zheng), and
scgong@zafu.edu.cn(S.C. Gong).}
\\{\small \it  School of Science, Zhejiang University of Science and Technology, }\\{\small \it
Hangzhou, 310023, P. R. China}
%\\{\small  \it b.School of Mathematical Sciences, Anhui University, Hefei 230601, P. R. China}
   }
\date{}
\maketitle

\baselineskip=0.20in

\noindent {\bf Abstract.}
An induced matching  in a graph is a set of edges whose
endpoints induce a $1$-regular subgraph.
Gupta et al. (2012,\cite{Gupta})
showed that every $n$-vertex graph has at most $10^{\frac{n}{5}}\approx 1.5849^n$ maximal induced matchings, which is  attained by the disjoint union of copies of the complete graph $K_5$.

In this paper, we establish the maximum number of maximal and maximum induced matchings in a connected graph of order  $n$ as follows:
\begin{align*}
\begin{cases}
{n\choose 2} &~ {\rm if}~ 1\leq n\le 8; \\
{{\lfloor \frac{n}{2} \rfloor}\choose 2}\cdot {{\lceil \frac{n}{2} \rceil}\choose 2} -(\lfloor \frac{n}{2} \rfloor-1)\cdot (\lceil \frac{n}{2} \rceil-1)+1 &~ {\rm if}~ 9\leq n\le 13; \\
10^{\frac{n-1}{5}}+\frac{n+144}{30}\cdot 6^{\frac{n-6}{5}} &~ {\rm if}~ 14\leq n\le 30;\\
10^{\frac{n-1}{5}}+\frac{n-1}{5}\cdot 6^{\frac{n-6}{5}} & ~ {\rm if}~ n\geq 31, \\
\end{cases}
\end{align*}
and we also demonstrate that this bound is tight.
This result implies that we can enumerate all  maximal induced matchings of an $n$-vertex connected graph  in time $O(1.5849^n)$.
Furthermore, our findings provide an estimate for the number of maximal dissociation sets in an
$n$-vertex  connected graph.
\vspace{3mm}



\baselineskip=0.20in

\section{Introduction}
All graphs considered in this paper are   finite, simple and undirected.
We follow the terminology
in \cite{Bondy}. Let $G=(V,E)$ be a simple graph.
A subset of vertices in $G$ is called an {\it independent set} ({\it dissociation set})
if it induces a subgraph with vertex degree $0$ (at most $1$).
An independent set (dissociation set) in $G$ is {\it maximal} if it is not a proper subset of any
other independent sets  (dissociation sets), or {\it maximum} if it has maximum cardinality.
A matching $M\subseteq E$ is
called {\it induced} if the endpoints of $M$ induce a $1$-regular subgraph in $G$.
Similarly, an induced matching is called {\it maximal} or {\it maximum} if it is not a proper subset of any other induced matchings or it has
maximum cardinality, respectively.
Clearly, dissociation set  encompasses the extremes of vertices with degree $0$ and degree $1$, serving as natural generalizations of both independent sets and induced matchings.

Let ${\mathcal G}_n$ be the set of all graphs with $n$ vertices. For a given graph $G$, denote by
$|MD_G|$ the number of maximal dissociation sets and by $|M_G|$ the number of maximal induced matchings within
$G$. For any two graphs $G$ and $H$, the notation $G\cup H$ signifies the disjoint union of $G$ and $H$,
and $sG$ denotes the disjoint union of $s$ copies of $G$.

Problems concerning the bounding of discrete structures that satisfy specific properties have garnered significant interest in combinatorics and graph theory.
In the 1960s, Erd\H{o}s and Moser posed the problem of determining the maximum number of maximal independent sets in general graphs with a fixed order, as well as identifying the graphs that achieve this maximum value.
Soon after, Erd\H{o}s, and   Moon and  Moser \cite{Moon} independently resolved both questions, demonstrating that the extremal graph is the disjoint union of $K_3$.
Wilf \cite{Wilf} was the first to consider these questions for connected graphs.
Subsequently, Griggs, Grinstead and Guichard \cite{Griggs}  provided a complete solution to the question.
Since then, the topic has been extensively studied across various classes of graphs, including trees, forests, (connected) graphs with at most one cycle, bipartite graphs, connected graphs, $k$-connected graphs and triangle-free graphs. For a comprehensive survey, see \cite{Jou}.

The concept of dissociation sets was introduced by Yannakakis \cite{Yannakakis} in 1978.
In 2011, Kardo\v{s} et al. established that the problem of finding a maximum dissociation set in a given graph is equivalent to the minimum $3$-path vertex cover problem, see  \cite{Bre, Kar} for details.
Recently, there has been extensive research on determining the maximum and minimum numbers of maximal dissociation sets within various families of graphs, along with identifying the corresponding extremal graphs \cite{Sun, Tu, Tu1, Zhang}. Notably, Tu et al. \cite[Theorem 5.1]{Tu} provided an upper bound for the number of maximal dissociation sets and characterized the associated extremal graphs; see Proposition  \ref{01}.

Induced matchings, also referred to as $2$-separated matchings, see Stockmeyer and Vazirani \cite{Stockmeyer}, and strong matchings, see Fricke and Laskar \cite{Fricke}.
In 1982, Stockmeyer and Vazirani \cite{Stockmeyer} first  studied the maximum induced
matching problem, (i.e.,  find a maximum induced matching, abbreviated as MIM)  as a generalization of the traditional graph matching problem.
This research has attracted considerable attention due to its expanding range of applications, as evidenced by various studies \cite{Cameron,Stockmeyer, Golumbic, Golumbic1,4,12,13}.
Recently, Xiao and  Tan \cite{Xiao} developed exact algorithms for solving the MIM problem.
Relatively speaking, there has been relatively little research on  maximizing $|M_G|$ over different families of graphs.
A celebrated result by Gupta, Raman, and Saurabh \cite{Gupta} stated that for a fixed $r$, there exists a constant $c<2$ such that every graph of order  $n$ has at most $c^{n}$
maximal $r$-regular induced subgraphs. Especially, for $r=1$, they showed that
\begin{proposition}\label{Gupta0}{\em  \cite[Theorem 4]{Gupta}}
Let $G$ be a (not necessarily connected) graph of order $n$. Then  $|M_G|\leq 10^{\frac{n}{5}}$
with equality iff $n\equiv 0~ (\bmod~5)$ and $G= \frac{n}{5}K_5$.
\end{proposition}
\noindent Subsequently, Manu Basavaraju,  Pinar Heggernes, et al. \cite{Basavaraju} put this counting problem within the framework of triangle-free graphs.

Motivated by Propositions  \ref{Gupta0}, \ref{01} and the corresponding research on maximal independent sets, the following question is natural:
\begin{ques} \label{02} Let $n$ be an arbitrary integer.  Which graph attains  the maximum
 number of maximal $(maximum)$ dissociation sets $($induced matchings$)$ among all connected graphs in ${\mathcal G}_n?$
\end{ques}

Let $G$ denote a disconnected graph whose all components are complete graphs; and
for a complete graph  $K_r$ disjoint from $G$, let $K_r\ast G$ be a connected graph obtained by picking a vertex $u\in V(K_r)$ and connecting it to exactly one vertex of each component in $G$, see Figure \ref{fig1} for an example.
In this paper, we consider Question \ref{02} on maximal and maximum induced matchings, showing that
\begin{theorem} \label{main1}
Let $G$ be a connected graph of order $n$. Then $|M_G|\leq f(n)$
with equality if $G=F_n$, where
 \begin{align*}
f(n):=
\begin{cases}
{n\choose 2} &~ {\rm if}~ 1\leq n\le 8; \\
{{\lfloor \frac{n}{2} \rfloor}\choose 2}\cdot {{\lceil \frac{n}{2} \rceil}\choose 2} -(\lfloor \frac{n}{2} \rfloor-1)\cdot (\lceil \frac{n}{2} \rceil-1)+1 &~ {\rm if}~ 9\leq n\le 13; \\
10^{\frac{n-1}{5}}+\frac{n+144}{30}\cdot 6^{\frac{n-6}{5}} &~ {\rm if}~ 14\leq n\le 30;\\
10^{\frac{n-1}{5}}+\frac{n-1}{5}\cdot 6^{\frac{n-6}{5}} & ~ {\rm if}~ n\geq 31, \\
\end{cases}
\end{align*}
and
\begin{align*}
F_n:=
\begin{cases}
K_n &~ {\rm if}~ 1\leq n\le 7; \\
K_8~{\rm or}~ K_4\ast K_4 &~ {\rm if}~ n=8; \\
K_{\lfloor \frac{n}{2}\rfloor}\ast K_{\lceil \frac{n}{2} \rceil} &~ {\rm if}~ 9\leq n\le 13; \\
K_6\ast \frac{n-6}{5}K_5 &~ {\rm if}~n \equiv 1 ~({\rm mod}~5)~ {\rm and}~ 14\leq n\le 30;\\
K_1\ast \frac{n-1}{5}K_5 &~  {\rm if}~n \equiv 1 ~({\rm mod}~5)~ {\rm and}~ n\geq 31. \\
\end{cases}
\end{align*}
  \end{theorem}
Denote $|MI_G|$  by  the number of  maximum induced matchings in a graph $G$.
Since a  maximum induced matching is   a maximal induced matching, $|MI_G|\leq |M_G|$. Note that the extremal  graph $F_n$ in Theorem  \ref{main1} satisfies that  each maximal induced matching is also  a  maximum induced matching, where  $n\leq 8$ and  $n\geq 31$. This implies the following result.
 \begin{theorem}
Let $G$ be a connected graph of order $n$. Then $|MI_G|\leq f(n)$
with equality if $G=F'_n$, where
 \begin{align*}
f(n):=
\begin{cases}
{n\choose 2} &~ {\rm if}~ 1\leq n\le 8; \\
{{\lfloor \frac{n}{2} \rfloor}\choose 2}\cdot {{\lceil \frac{n}{2} \rceil}\choose 2} -(\lfloor \frac{n}{2} \rfloor-1)\cdot (\lceil \frac{n}{2} \rceil-1)+1 &~ {\rm if}~ 9\leq n\le 13; \\
10^{\frac{n-1}{5}}+\frac{n+144}{30}\cdot 6^{\frac{n-6}{5}} &~ {\rm if}~ 14\leq n\le 30;\\
10^{\frac{n-1}{5}}+\frac{n-1}{5}\cdot 6^{\frac{n-6}{5}} & ~ {\rm if}~ n\geq 31, \\
\end{cases}
\end{align*}
and
\begin{align*}
F'_n:=
\begin{cases}
K_n &~ {\rm if}~ 1\leq n\le 8; \\
K_1\ast \frac{n-1}{5}K_5 &~  {\rm if}~n \equiv 1 ~({\rm mod}~5)~ {\rm and}~ n\geq 31. \\
\end{cases}
\end{align*}
  \end{theorem}
\begin{figure}[ht!]
\begin{center}
\begin{tikzpicture}[scale=0.9,style=thick]
\tikzstyle{every node}=[draw=none,fill=none]
\def\vr{3pt} % \vr = vertex radius;  Set \vr = 2/scale for uniform sizing of vertices

\begin{scope}[yshift = 0cm, xshift = 0cm]
%% vertices defined %%
\path (0,0) coordinate (x1);
\path (-1,-1.5) coordinate (x2);
\path (-2.85,-1.5) coordinate (x3);
\path (2.65,-1.5) coordinate (x4);

%% edges %%
\draw (x1) --(x2);
\draw (x1) --(x3);
\draw (x1) --(x4);
\draw(0,0.6) circle[radius=0.6];
\draw(-1,-2.1) circle[radius=0.6];
\draw(-3,-2.1) circle[radius=0.6];
\draw(2.8,-2.1) circle[radius=0.6];

\draw (x1)  [fill=black] circle ;
\draw (x2)  [fill=white] circle ;
\draw (x3)  [fill=white] circle ;
\draw (x4)  [fill=white] circle ;
%\draw (x4)  [fill=white] circle (\vr);
%% text %%
%\draw[above] (x1)++(0.5,-0.3) node {$u$};
\draw[above] (x1)++(0,0.3) node {$K_r$};
%\draw[above] (x2)++(0.5,-0.2) node {$v_2$};
\draw[above] (x2)++(0,-0.9) node {$K_{s_2}$};
%\draw[above] (x3)++(0.5,-0.2) node {$v_1$};
\draw[above] (x3)++(-0.15,-0.9) node {$K_{s_1}$};
%\draw[above] (x4)++(0.5,-0.2) node {$v_t$};
\draw[above] (x4)++(0.15,-0.9) node {$K_{s_t}$};
\draw[above] (x2)++(1.95,-0.9) node {$\cdots$};

% Simple brace
% \draw [decorate, decoration = {brace}] (0.9,0.2) --  (3.1,0.2);
\end{scope}
\end{tikzpicture}
\end{center}
\caption{ $K_r\ast (K_{s_1}\cup K_{s_2}\cup \cdots \cup K_{s_t})$.}
\label{fig1}
\end{figure}
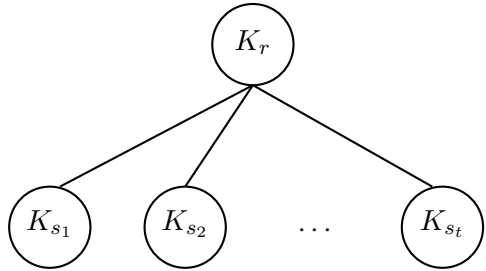

We will prepare some preliminaries in Section $2$ and present the proof of Theorem \ref{main1}
in Section $3$.

\section{Preliminaries\label{Preliminaries}}
%\subsection{\textbf{Definitions and notations}}

Let $G$ be a graph and   $S\subseteq V(G)$. Denote by $G[S]$ the subgraph of $G$  induced by $S$, and denote by $G-S$ the induced subgraph  $G[V(G)\setminus S]$.  In particular, if $S=\{v\},$ then we write $G-v$ instead of
$G-\{v\}$.
For a vertex $v\in V(G)$, the  closed and open neighborhoods of $v$ is written as  $N_G[v]$ and $N_G(v)$, respectively.
For a subset $S\subset  V(G)$, the closed and open  neighborhoods of $S$ is defined as $N_G[S]=\cup_{v\in S}N_G[v]$ and $N_G(S)=N_G[S]\setminus S$, respectively.

For an induced matching $M$ of $G$ and a vertex $v \in V(G)$, we say that $M$ covers $v$ if $v$ is an endpoint
of an edge in $M$. Denote $V(M)$  by the set of vertices covered in $M$.
For two disjoint subsets $S, T\subseteq V(G)$, let $M_G(S,T)$ be the set of all maximal induced matchings of $G$ that cover no vertex
of $S$ and every vertex of $T$. Clearly, $M_G=M_G(\emptyset,\{v\})\uplus M_G(\{v\},\emptyset)$ holds for each vertex $v$ of $G$, where $\uplus$ is the disjoint union hereafter. If there is no confusion, we
omit subscripts from the notations.

Let $u, v$ be two vertices of a graph $G$. They are said to be {\it twins} ({\it false twins}) if $N_G[u]=N_G[v]$ ($N_G(u)=N_G(v)$).
Define the twin set of $v$  as $T_G(v) = \{u \in V(G)~|~N_G[u]=N_G[v]\}$,
i.e., $T_G(v)$ consists of the vertex $v$ and all its twins.
Write $\tau(G)$ as the number of twin sets in  $G$.
When $uv\in E(G),$ define $G_{v\rightarrow u}$ to be the graph obtained from $G$ by deleting the
edge $vx$ for every $x\in N(v)\setminus N[u]$ and adding the edge $vy$ for every $y\in N(u)\setminus N[v]$, i.e., making  $v$ into a twin of $u$. See Figure   \ref{fig2} for an example.
Further, define  $G_{T_{G}(v)\rightarrow u}$ to be the graph obtained from $G$ by making each vertex of $T_G(v)$ into a twin of $u$.
\begin{figure}[ht!]
\begin{center}
\begin{tikzpicture}[scale=0.9,style=thick]
\tikzstyle{every node}=[draw=none,fill=none]
\def\vr{3pt} % \vr = vertex radius;  Set \vr = 2/scale for uniform sizing of vertices

\begin{scope}[yshift = 0cm, xshift = 0cm]
%% vertices defined %%
\path (-4.8,1.1) coordinate (u);
\path (-3.2,1.1) coordinate (v);
\path (-4.4,0) coordinate (x1);
\path (-3.6,0) coordinate (x2);
\path (-5.2,0) coordinate (x3);
\path (-2.8,0) coordinate (x4);
\path (-6,0) coordinate (x5);
\path (-2,0) coordinate (x6);
\path (3.2,1.1) coordinate (u1);
\path (4.8,1.1) coordinate (v1);
\path (3.6,0) coordinate (x7);
\path (4.4,0) coordinate (x8);
\path (2.8,0) coordinate (x9);
\path (5.2,0) coordinate (x10);
\path (2,0) coordinate (x11);
\path (6,0) coordinate (x12);
%% edges %%
\draw (u) --(x1);
\draw (u) --(x2);
\draw (u) --(x3);
\draw (u) --(x5);
\draw (v) --(x1);
\draw (v) --(x2);
\draw (v) --(x4);
\draw (v) --(x6);
\draw (u1) --(x7);
\draw (u1) --(x8);
\draw (u1) --(x9);
\draw (u1) --(x11);
\draw (v1) --(x7);
\draw (v1) --(x8);
\draw (v1) --(x9);
\draw (v1) --(x11);
\draw (u) --(v);
\draw (u1) --(v1);
%\draw(0,0.6) circle[radius=0.6];
%\draw(-1,-2.1) circle[radius=0.6];
%\draw(-2.7,-2.1) circle[radius=0.6];
%\draw(2.5,-2.1) circle[radius=0.6];
\draw (-4,-0.45) ellipse[x radius=2.8,y radius=1.1];
\draw (4,-0.45) ellipse[x radius=2.8,y radius=1.1];

\draw (u)  [fill=white] circle(\vr) ;
\draw (v)  [fill=white] circle(\vr) ;
\draw (x1)  [fill=white] circle(\vr) ;
\draw (x2)  [fill=white] circle(\vr) ;
\draw (x3)  [fill=white] circle(\vr) ;
\draw (x4)  [fill=white] circle(\vr) ;
\draw (x5)  [fill=white] circle(\vr) ;
\draw (x6)  [fill=white] circle(\vr) ;
\draw (u1)  [fill=white] circle(\vr) ;
\draw (v1)  [fill=white] circle(\vr) ;
\draw (x7)  [fill=white] circle(\vr) ;
\draw (x8)  [fill=white] circle(\vr) ;
\draw (x9)  [fill=white] circle(\vr) ;
\draw (x10)  [fill=white] circle(\vr) ;
\draw (x11)  [fill=white] circle(\vr) ;
\draw (x12)  [fill=white] circle(\vr) ;

\draw[above] (u)++(0,0.1) node {$u$};
\draw[above] (v)++(0,0.1) node {$v$};
\draw[above] (u1)++(0,0.1) node {$u$};
\draw[above] (v1)++(0,0.1) node {$v$};

% Simple brace
% \draw [decorate, decoration = {brace}] (0.9,0.2) --  (3.1,0.2);
\end{scope}
\end{tikzpicture}
\end{center}
\caption{Graphs $G$ (left) and  $G_{v\rightarrow u}$ (right).}
\label{fig2}
\end{figure}
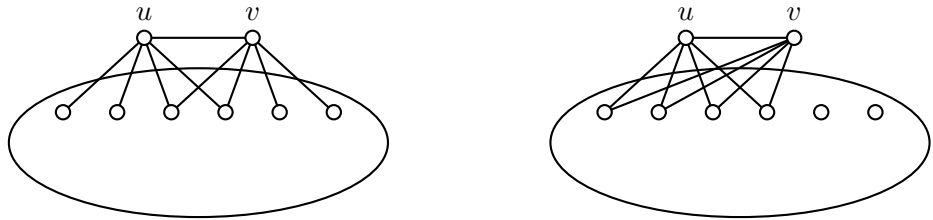

In the following, we give a condition such that the resulting graph $G_{v\rightarrow u}$ increases the number of maximal induced matchings.

%Naturally, it is interesting to  identify which   pairs of vertices do  not decrease the number of maximal induced matchings.
\begin{lemma} \label{y1}
Let $G$ be a graph and $uv \in E(G)$. Let $W_1=N(u)\setminus N[v]$, $W_2=N(u)\cap N(v)$, $W_3=N(v)\setminus N[u]$, and $W_4=V(G)\setminus (N[u]\cup N[v])$. If there is no maximal induced matching $M$ of $G$ such that $V(M)\subseteq W_1\cup W_4$ with $V(M)\cap W_1\neq\emptyset$ and $V(M)\subseteq W_3\cup W_4$ with $V(M)\cap W_3\neq\emptyset$, then  $|M_{G_{u\rightarrow v}}|\geq|M_G|$ or $|M_{G_{v\rightarrow u}}|\geq|M_G|$.
\end{lemma}
\begin{proof}  Suppose that  $|M_G(\emptyset, \{u\})| \geq  |M_G(\emptyset, \{v\})|$. Since $ M_G(\emptyset,\{u\})=M_G(\emptyset,\{u, v\})\uplus M_G(\{v\},\{u\})$ and $M_G(\emptyset,\{v\})=M_G(\emptyset,\{u, v\})\uplus M_G(\{u\},\{v\})$, it follows that $|M_G(\{v\}, \{u\})| \geq  |M_G(\{u\}, \{v\})|$. For convenience, we write $H=G_{v\rightarrow u}$.
The sets $M_G$ and  $M_{H}$ can be partitioned, respectively, as follows:
$$M_G =  M_G(\{v\},\{u\})\uplus M_G(\{u\},\{v\}) \uplus M_G(\emptyset,\{u, v\}) \uplus M_G(\{u, v\}, \emptyset)$$
and
$$M_{H} =  M_{H}(\{v\},\{u\})\uplus M_{H}(\{u\},\{v\}) \uplus M_{H}(\emptyset,\{u, v\}) \uplus M_{H}(\{u, v\}, \emptyset).$$

We first claim that $|M_G(\{v\},\{u\})|=|M_{H}(\{v\},\{u\})|$. Note that
$$|M_G(\{v\},\{u\})|=\sum_{p\in N_G(u)\setminus \{v\}}|M_G(\{v\},\{u,p\}|=\sum_{p\in N_G(u)\setminus \{v\}}|M_{G-N_G[u]\cup N_G[p]}|$$
and  similarly  $|M_H(\{v\},\{u\})|=\sum_{p\in N_H(u)\setminus \{v\}}|M_{H-N_H[u]\cup N_H[p]}|$. It is easy to see that for every $p\in N_G(u)\setminus
\{v\}$ and its corresponding $p\in N_H(u)\setminus \{v\}$, $G-N_G[u]\cup N_G[p]= H-N_H[u]\cup N_H[p].$  Thus the  claim follows.

 Then  $|M_{H}(\{u\},\{v\})|=|M_{H}(\{v\},\{u\})|=|M_{G}(\{v\},\{u\})|\geq |M_{G}(\{u\},\{v\})|$, by the assumption that $|M_G(\emptyset, \{u\})| \geq  |M_G(\emptyset, \{v\})|$ and the fact that $u,v$ are twins in $H$.

Next,  $|M_G(\emptyset,\{u, v\})|\leq |M_H(\emptyset,\{u, v\})|$, since $G[W_4]$ is an induced subgraph of $H[W_3\cup W_4]$, and
 $|M_G(\emptyset,\{u, v\})|=|M_{G-N_G[u]\cup N_G[v]}|=|M_{G[W_4]}|$ and $|M_H(\emptyset,\{u, v\})|=|M_{H-N_H[u]\cup N_H[v]}|=|M_{H[W_3\cup W_4]}|$.

Finally, we claim that $|M_G(\{u, v\}, \emptyset)|\leq |M_H(\{u, v\}, \emptyset)|$.
Let $M \in M_G(\{u, v\}, \emptyset)$ (note that $V(M)\cap (W_1\cup W_2\cup W_3)\neq \emptyset$). Now we prove that $M\in M_H(\{u, v\}, \emptyset)$. It is clear that $M$ is an induced matching of $H$, as   the only edges that are modified are incident with $v$ and $M$ does not cover $v$.
For contradiction, suppose that $M$ is not a maximal induced matching of $H$.
It implies that there is an edge $xy\in E(H)$ such that $M \cup \{xy\}$ is an induced matching of $H$.
 Then we divide into three cases to consider the possibilities of $x, y$.
In graph $H$, if $xy=uv$, then $V(M)\subseteq W_3\cup W_4$. Further, it implies that there is a $M\in M_G$ such that
$V(M)\subseteq W_3\cup W_4$ with $V(M)\cap W_3\neq\emptyset$, contradicting the assumption.
If $v \notin \{x, y\}$ and $u \in \{x, y\}$ or $u,v \notin \{x, y\}$, then $xy\in E(G)$ and thus  $M \cup \{xy\}$ is also an induced matching of $G$ since we only change edges incident with $v$ when transforming $G$ into $H$ and $M$ does not cover $v$. This contradicts the maximality of $M$.
If $v \in\{x, y\}$ and $u \notin \{x, y\}$,
then we may assume  that $x =v$.
Let $M{'}= M \cup \{vy\}$. Since $M'$ is an induced matching and $u$ and $v$ are twins in $H$, we infer that $M{''}= M \cup \{uy\}$ is also an induced matching of $H$. Note that the edge $uy$ is also present in $G$, so $M{''}$ is an induced matching of $G$. It contradicts the maximality of $M$.
This completes the proof of Lemma \ref{y1}.
%Therefore,  the proof is complete.
\end{proof}

\begin{remark}  The constraint conditions in Lemma  \ref{y1} are necessary.
%The conditions in Lemma  \ref{y1} can not be omitted.
Consider $G=C_5=v_1v_2v_3v_4v_5v_1$. Then
$|M_G|=5> |M_{G_{v_1\rightarrow v_2}}|=|M_{G_{v_2\rightarrow v_1}}|=4.$
\end{remark}

Repeating the above operation on all the vertices in $T_{G}(v)$, we  have
\begin{lemma} \label{y-y1}
Let $G$ be a graph and $uv \in E(G)$. Let $W_1=N(u)\setminus N[v]$, $W_2=N(u)\cap N(v)$, $W_3=N(v)\setminus N[u]$, and $W_4=V(G)\setminus (N[u]\cup N[v])$. If there is no maximal induced matching $M$ of $G$ such that $V(M)\subseteq W_1\cup W_4$ with $V(M)\cap W_1\neq\emptyset$ and $V(M)\subseteq W_3\cup W_4$ with $V(M)\cap W_3\neq\emptyset$, then  $|M_{G_{T_{G}(u)\rightarrow v}}|\geq|M_G|$ or $|M_{G_{T_{G}(v)\rightarrow u}}|\geq|M_G|$.
\end{lemma}

Similar to the proof of Lemma 4 in \cite{Basavaraju}, we   have the following result.
\begin{lemma} \label{twins}
Let $G$ be a graph and $uv \in E(G)$. If $u$ and $v$ are not twins in $G$, then $\tau(G_{u\rightarrow v})\leq \tau(G)$ and $\tau(G_{v\rightarrow u})\leq \tau(G)$. In particular,   $\tau(G_{v\rightarrow u})< \tau(G)$ if $|T_{G}(v)|=1$; and $\tau(G_{v\rightarrow u})= \tau(G)$ if $|T_{G}(v)|\geq 2$; and      $\tau(G_{T_{G}(v)\rightarrow u})< \tau(G)$.
\end{lemma}

Next, for any given graph $G$ and a vertex $v\in V(G)$, we aim to establish the relationship between
$M_{G}$ and $M_{G-v}$. We partition  $M_G(\emptyset,\{v\})$ and $M_{G-v}$
into two distinct sets as follows, and these notations will be frequently utilized in subsequent discussions.
A maximal induced matching $M\in M_G(\emptyset,\{v\})$ is referred to as {\it $A$-type} with respect to vertex
$v$  if the induced matching obtained by deleting the edge containing $v$
is no longer a maximal induced matching in the graph $G-v$. We denote the number of all $A$-type maximal induced matchings of $G$ with respect to vertex $v$ as $\alpha(v;G)$.
Additionally, we denote $\beta(v;G)$ as the count of maximal induced matchings $X\in M_G(\emptyset,\{v\})$
such that when the edge incident with $v$ is removed from
$X$, the resulting matching remains maximal in  $G-v$.
Clearly,  $|M_G(\emptyset,\{v\})|=\alpha(v;G)+\beta(v;G)$.

Let $M_{G-v}:=\Theta(v;G)\cup \Phi(v;G),$ where $\Theta(v;G):=\{~M~|~M\in M_{G-v}~\text{and}~M\in M_{G}\}$ and
$\Phi(v;G):=\{~M~|~M\in M_{G-v}~\text{and}~M\notin M_{G}\}$. Denote by $\theta(v;G)=|\Theta(v;G)|$ and $\varphi(v;G)=|\Phi(v;G)|.$
Observe that every $M\in \Phi(v;G)$ can expand to be a maximal induced matching  of $G$ by adding some edge incident with $v$. This implies that $\varphi(v;G)\leq \beta(v;G).$
In particular, if for every $M\in \Phi(v;G)$, there is exactly one edge $vp\in E(G)$ such that $M\cup \{vp\}\in M_G(\emptyset,\{v\})$, then $\varphi(v;G)=\beta(v;G).$
It can readily be verified that
$\theta(v;G)=|M_G(\{v\},\emptyset)|$. From these relations and notations,  it is easy to establish the following result.
\begin{proposition}  \label{bopro1}
Let $G$ be a  graph with $v\in V(G)$. Then
$$|M_{G-v}|+\alpha(v;G)\leq |M_G|\leq |M_{G-v}|+\sum_{p\in N(v)}|M_{G-N[v]\cup N[p]}|.$$
Furthermore, the first inequality holds  if and only if $\beta(v;G)=\varphi(v;G);$
and  the second inequality holds  if and only if $\varphi(v;G)=0.$
\end{proposition}
\begin{proof} Following the definitions of $\alpha(v;G), \beta(v;G), \Theta(v;G),$ and $\Phi(v;G)$ above, we have
%$$|M_{G-v}|+ \sum_{p\in N(v)}|M_{G-N[v]\cup N[p]}|=|M_{G-v}|+\sum_{p\in N(v)}|M_G(\emptyset,\{v,p\}|\geq
% |M_G|=|M_G(\{v\},\emptyset)|+|M_G(\emptyset,\{v\})|=|M_G(\{v\},\emptyset)|+\beta(v;G)+\alpha(v;G)\geq \theta(v;G)+\varphi(v;G)+\alpha(v;G)=|M_{G-v}|+\alpha(v;G).$$
\begin{equation*}%\label{boi4}
\begin{aligned}
|M_G|& =|M_G(\{v\},\emptyset)|+|M_G(\emptyset,\{v\})|\\
&=|M_G(\{v\},\emptyset)|+\alpha(v;G)+\beta(v;G)\\
&\geq \theta(v;G)+\alpha(v;G)+\varphi(v;G)\\
&=|M_{G-v}|+\alpha(v;G)
\end{aligned}
\end{equation*}
and   equality holds if and only if $\beta(v;G)=\varphi(v;G).$

The upper bound on $|M_G|$ follows from
\begin{equation*}
\begin{aligned}
|M_G|& =|M_G(\{v\},\emptyset)|+|M_G(\emptyset,\{v\})|\\
&\leq |M_{G-v}|+\sum_{p\in N(v)}|M_G(\emptyset,\{v,p\}|\\
&= |M_{G-v}|+ \sum_{p\in N(v)}|M_{G-N[v]\cup N[p]}|
\end{aligned}
\end{equation*}
and equality holds if and only if $|M_G(\{v\},\emptyset)|=|M_{G-v}|.$
Recall that $|M_G(\{v\},\emptyset)|=\theta(v;G)$ and $|M_{G-v}|=\theta(v;G)+\varphi(v;G)$.
The result follows.
\end{proof}
In Proposition \ref{bopro1}, we emphasize that in the expression  $\sum_{p\in N(v)}|M_{G-N[v]\cup N[p]}|,$
$|M_{G-N[v]\cup N[p]}|=1$ if $G-N[v]\cup N[p]$ contains no edges, as it  represents the contribution of each edge $vp$ to $M_G(\emptyset,\{v\})$.

\begin{remark} In Proposition  \ref{bopro1}, since $\alpha(v;G)\geq 0,$ $|M_G|\geq |M_{G-v}|.$ Further, $|M_G|\geq |M_H|$ if $H$ is an (vertex-)induced subgraph of $G.$
But, in general, Proposition  \ref{bopro1} cannot hold if $H$  is an edge-induced subgraph of  $G.$
See an example:~$|M_{K_4\ast K_4}|=|M_{K_8}|=28>|M_H|$, where $H$  is an any connected graph of order $8$
except for $K_4\ast K_4$, $K_8.$
\end{remark}

Next we consider some special cases of Proposition  \ref{bopro1} for later use.
Since $\varphi(v;G)\leq \beta(v;G),$ $\beta(v;G)=0$ implies the following corollary.
\begin{corollary}\label{inhold0}
Let $G$ be a  graph with $v\in V(G)$. If $\beta(v;G)=0$ $($or equivalently $\alpha(v;G)=|M_G(\emptyset,\{v\})|$$)$, then  $|M_G|=|M_{G-v}|+\alpha(v;G)=|M_{G-v}|+\sum_{p\in N(v)}|M_{G-N[v]\cup N[p]}|.$
\end{corollary}

A block   of $G$ is a maximal $2$-connected subgraph, and it is said to be a
{\it pendant block} of $G$ if it has exactly one cutpoint in $G$ as a whole.
As a consequence of Corollary \ref{inhold0}, we have
\begin{corollary}\label{bocor2}
If $G$ contains a pendant block $K_r~(r\geq 3)$ with cutpoint $v$, then
$$|M_G|={{r-1}\choose 2}\cdot |M_{G-K_r}|+\sum_{p\in N(v)}|M_{G-N[v]\cup N[p]}|.$$
\end{corollary}
\begin{proof} Since $v$ is a cutpoint, $|M_{G-v}|=|M_{K_{r-1}}|\cdot |M_{G-K_r}|={{r-1}\choose 2}\cdot |M_{G-K_r}|.$ Now we would claim that  each $M\in M_G(\emptyset,\{v\})$ is of $A$-type, and thus
$\alpha(v;G)=|M_G(\emptyset,\{v\})|$. Then the result follows by Corollary  \ref{inhold0}.
%$=\sum_{p\in N(v)}|M_G(\emptyset,\{v,p\}|=\sum_{p\in N(v)}|M_{G-N[v]\cup N[p]}|.$
For each $M\in M_G(\emptyset,\{v\})$, denote by $M'$ the induced matching obtained from  $M$ by deleting the edge incident with $v$.
Observe that $M'$ is not a maximal induced matching  of $G-v$, since $M'\cup \{xy\}$ is an induced matching  of $G-v$ where  $xy\in E(K_r-v)$. So according to the definition, $M$ is of $A$-type.
\end{proof}

%\subsection{\textbf{Relations between $f(n)$ and $g(n)$}}
Due to Gupta, Raman, and Saurabh \cite{Gupta}, the maximum value of maximal induced matchings in  graphs of order $n$ is determined.
They show that, for any (not necessarily connected) graph $G$ of  order $n$,
$|M_G|\leq 10^{\frac{n}{5}}$
with equality if and only if $n\equiv 0~ (\bmod~5)$ and $G= \frac{n}{5}K_5$, see Proposition \ref{Gupta0}.
This result  plays a crucial role in our discussion. For convenience, set
%$$g(n):=10^{\frac{n}{5}}.\eqno{(2.1)}$$
\begin{equation}\label{(2.1)}
%\tag{$\star$}
\begin{aligned}
 g(n):=10^{\frac{n}{5}}.
\end{aligned}
\end{equation}

 Henceforth, $f(n)$ and $g(n)$ are  functions on $n$ defined in Theorem \ref{main1} and Eq.~(\ref{(2.1)}), respectively.
 In the end of this section, we will establish a series of conclusions related to the relationship between functions $f(n)$ and $g(n)$.

\begin{lemma}\label{gong3}
Let  $p(x)=(x-1)g(n-1-x)+g(n-2-x)$, where $2\le x\le n-2$. Then $p(x)$ is a decreasing function on $x$ when $x\ge 3$, and $$p(3)=\max\{~p(x)~|~x\in \mathbb{Z^+}, 2\le x\le n-2\}.$$
\end{lemma}
\begin{proof} Taking derivative to the  function $p(x)$,  we have
$$p'(x)=10^{\frac{n-x-2}{5}}\cdot(10^{\frac{1}{5}}-\frac{(x-1)\cdot 10^{\frac{1}{5}}\cdot \ln 10}{5}-\frac{\ln 10}{5}).$$
In view of $p'(x)<0$ for $x\ge 3$, then $p(x)$ is  decrease on $x\ge 3$. Thus the result follows by the fact that $p(3)>p(2)$.
\end{proof}

For convenience, when $n\geq14$, the piecewise function $f(n)$ can be written as $10^{\frac{n-1}{5}}+f_1(n)$, where
\begin{equation}\label{eq}
%\tag{$\star$}
\begin{aligned}
 f_1(n)=
\begin{cases}
  \frac{n+144}{30}\cdot6^{\frac{n-6}{5}}&~ $if$~\mbox{$14\leq n\leq30$};\\
  \frac{n-1}{5}\cdot 6^{\frac{n-6}{5}}&~ $if$~\mbox{$n\geq31$}.
\end{cases}
\end{aligned}
\end{equation}
It is easy to check that,  for  $n\geq14$, $f_1(n)$ satisfies:\\
{\em $($$1$$)$}.	$f_1(n)$ is monotonic increase with respect to  $n$;\\
{\em $($$2$$)$}.  $f_1(n-1)+sf_1(n-t)<f_1(n)$, where $s=1$ and $t\geq4$ or $s=2$ and $t\geq6$ or $s=3$ and $t\geq7$.
%\begin{enumerate}[label=($\alph*$).]
%    \item
%      $f_1(n)$ is monotonic increase with respect to  $n$.
%    \item
%     $f_1(n-1)+sf_1(n-t)<f_1(n)$, where $s=1$ and $t\geq4$ or $s=2$ and $t\geq6$ or $s=3$ and $t\geq7$.
%\end{enumerate}
%For $n\geq14$, note that $f(n)=10^{\frac{n-1}{5}}+f_1(n)$, where
%\begin{equation}\label{eq}
%\tag{$\star$}
%\begin{aligned}
% f_1(n)=
%\begin{cases}
%  \frac{n+144}{30}\cdot6^{\frac{n-6}{5}}, & \mbox{$14\leq n\leq30$};\\
%  \frac{n-1}{5}\cdot 6^{\frac{n-6}{5}}, & \mbox{$n\geq31$}.
%\end{cases}
%\end{aligned}
%\end{equation}
%\begin{claim}\label{y7}
%For  $n\geq14$,  $f_1(n)$ is monotonic increase with respect to  $n$.
%\end{claim}
%\begin{claim}\label{y8}
%For  $n\geq14$, $f_1(n-1)+sf_1(n-t)<f_1(n)$, where $s=1$ and $t\geq4$ or $s=2$ and $t\geq6$ or $s=3$ and $t\geq7$.00000
%\end{claim}

In our proof, we need bound the upper bound on $|M_G|$  in terms of  $\frac{f(n)}{g(n)}$ with different forms and so  we   give a list of inequalities between $f(n)$ and $g(n)$ as follows which will be frequently used.
\begin{lemma}\label{botool1}  Let  $n$ be a positive integer and  $f(n)$ and $g(n)$ be two functions defined as above. Then\\
{\em (1)}.	$\frac{f(n)}{g(n)}\le \frac{f(4)}{g(4)}< 0.9510$ if $n\neq 5$;\\
{\em (2)}.  $\frac{f(n)}{g(n)}\le \frac{f(6)}{g(6)}< 0.9465$ if $n\notin \{4,5\}$;\\
{\em (3)}.  $\frac{f(n)}{g(n)}\le \frac{f(10)}{g(10)}\leq 0.8500$ if $n\geq 7$;\\
%{\em (4)}. $\frac{f(n)}{g(n)}\le \frac{f(11)}{g(11)}< 0.8266$ if $n\geq 11$;\\
{\em (4)}. $\frac{f(n)}{g(n)}\le \frac{f(14)}{g(14)}< 0.7778$ if $n\geq 13$.
\end{lemma}
\begin{proof} By a direct verification, the result follows when $n\le 13$.
For $n\ge 14$, let \begin{align*}
h(n):=\frac{f(n)}{g(n)}=
\begin{cases}
10^{-\frac{1}{5}}+\frac{n+144}{5}\cdot 6^{-\frac{11}{5}}\cdot (\frac{3}{5})^{\frac{n}{5}}&~ $if$~ 14\leq n\le 30;\\
10^{-\frac{1}{5}}+\frac{n-1}{5}\cdot 6^{-\frac{6}{5}}\cdot (\frac{3}{5})^{\frac{n}{5}} & ~ $if$~ n\geq 31. \\
\end{cases}
\end{align*}
Let now $h_1(x)$ and $h_2(x)$ be two continuous functions on $x$ defined as
\begin{equation*}%\label{ac}
\begin{aligned}
h_1(x)&=  10^{-\frac{1}{5}}+\frac{x+144}{5}\cdot 6^{-\frac{11}{5}}\cdot (\frac{3}{5})^{\frac{x}{5}}~\text{for}~14\le x\le 30;\\
h_2(x)&=  10^{-\frac{1}{5}}+\frac{x-1}{5}\cdot 6^{-\frac{6}{5}}\cdot (\frac{3}{5})^{\frac{x}{5}}~\text{for}~ x\geq 30.
\end{aligned}
\end{equation*}
Taking derivative respectively of functions $h_1(x)$ and $h_2(x)$, we have
\begin{equation*}%\label{ac}
\begin{aligned}
h_1'(x)&=0.2\cdot 6^{-\frac{11}{5}}\cdot 0.6^\frac{x}{5}(1+0.2\cdot(x+144)\cdot\ln(0.6))<0~\text{for}~14\le x\le 30;\\
h_2'(x)&=0.2\cdot 6^{-\frac{6}{5}}\cdot 0.6^\frac{x}{5}(1+0.2\cdot(x-1)\cdot\ln(0.6))<0~\text{for}~ x\geq 30,
\end{aligned}
\end{equation*} which implies that both of functions $h_1(x)$ and $h_2(x)$ are  monotonic decrease.
Moreover, we find that $h_1(30)=h_2(30)$. Thus, for $x\ge 14$,
$$\frac{f(x)}{g(x)}\le \frac{f(14)}{g(14)}< 0.7778.$$
Thus, the proof is complete.
%Especially, when $n=10,$
%$\frac{f(10)}{g(10)}= \frac{85}{100}.$
\end{proof}

\begin{lemma} \label{bolemma0011} Let $n$ be an integer with $n\ge 14.$ Then\\
{\em (1).} when $r\in \{4,5,6\}$, ${{r-1}\choose 2}\cdot f(n-r)+(r-1)\cdot f(n-r-1)+g(n-r-2)<f(n)$;\\
{\em  (2).} when $r\in \{4,5,6\}$, ${{r-1}\choose 2}\cdot f(n-r)+r\cdot g(n-r-2)+g(n-r-3)<f(n)$;\\
{\em  (3).} when $r\in \{4,6\}$,  ${{r-1}\choose 2}\cdot f(n-r)+(r-1)\cdot g(n-r-1)+6^{\frac{n-r-1}{5}}<f(n)$;\\
{\em  (4).} when $r\in \{4,6\}$,  ${{r-1}\choose 2}\cdot f(n-r)+\frac{f(4)}{g(4)}\cdot (r-1)\cdot g(n-r-1)+g(n-r-3)<f(n)$.
\end{lemma}
\begin{proof}
$(1)$. We divide our proof into the following two cases:\\
{\bf Case 1. $r\in \{4,5\}$.}

For $n-r< 14$, the result follows by a direct calculation.
 When  $n-r\geq 14$, $\frac{f(n-r-1)}{g(n-r-1)}\leq \frac{f(14)}{g(14)}< 0.7778$ by    Lemma \ref{botool1} and thus $f(n-r-1)< 0.7778g(n-r-1)$.
 Then we only need to show that ${{r-1}\choose 2}\cdot f(n-r)+0.7778\cdot(r-1)\cdot g(n-r-1)+g(n-r-2)<f(n)$.
 By Eq.~(\ref{eq}),  the coefficient of $10^{\frac{n-1}{5}}$ in the left-hand side  is ${{r-1}\choose 2}\cdot 10^{-\frac{r}{5}}+(r-1)\cdot 0.7778\cdot 10^{-\frac{r}{5}}+10^{-\frac{r+1}{5}}<1$ for  each $r.$
Besides, it is easy to check that
${{r-1}\choose 2}\cdot f_1(n-r)\leq f_1(n)$ when   $n-r\geq 14$.

\noindent {\bf Case 2.} $r=6$.

For $n< 38$, the result follows by a direct calculation.
When  $n\geq 38$, $\frac{f(n-r-1)}{g(n-r-1)}\leq \frac{f(31)}{g(31)}< 0.6604$ by  Lemma \ref{botool1}, i.e. $f(n-r-1)< 0.6604g(n-r-1)$. Thus,
\begin{equation*}%\label{ac}
\begin{aligned}
&{{r-1}\choose 2}\cdot f(n-r)+(r-1)\cdot f(n-r-1)+g(n-r-2)\\
<~&{{r-1}\choose 2}\cdot0.6604\cdot g(n-r)+(r-1)\cdot0.6604\cdot g(n-r-1)+g(n-r-2)\\
=~ &(10\cdot 0.6604\cdot 10^{-\frac{5}{5}}+5\cdot 0.6604\cdot 10^{-\frac{6}{5}}+10^{-\frac{7}{5}})\cdot 10^{\frac{n-1}{5}}<0.9086\cdot 10^{\frac{n-1}{5}}<f(n).
\end{aligned}
\end{equation*}
%By a direct verification, the conclusion holds for the remaining cases  where  $14\leq n\leq 37.$

% where $f_1(n)$ is defined in Equation~(\ref{eq}). These imply that the conclusion holds for  $r\in\{4,5\}$ and $n-r\geq 14$.
%Consequently, by a direct calculation,  $(1)$ holds for  the remaining cases where $r\in \{4,5\}$ and $14\leq n\leq 13+r$.
%
%
%
%
%
%When  $n-r-1\geq 11$, $\frac{f(n-r-1)}{g(n-r-1)}\leq \frac{f(11)}{g(11)}< 0.8266$ by    Lemma \ref{botool1} and thus $f(n-r-1)< 0.8266g(n-r-1)$.
%By Eq.~(\ref{eq}),  the coefficient of $10^{\frac{n-1}{5}}$ in the left-hand side of $(1)$ is ${{r-1}\choose 2}\cdot 10^{-\frac{r}{5}}+(r-1)\cdot 0.8266\cdot
% 10^{-\frac{r}{5}}+10^{-\frac{r+1}{5}}<1$ holds for each  $r(r\in \{4,5,6\}).$
%Besides, it is easy to check that when  $r\in\{4,5\}$ and $n-r\geq 14$,
%${{r-1}\choose 2}\cdot f_1(n-r)\leq f_1(n)$ where $f_1(n)$ is defined in Eq.~(\ref{eq}). These imply that the conclusion holds for  $r\in\{4,5\}, n-r\geq 14$.
%Consequently, by,  $(1)$ holds for  the remaining cases where $r\in \{4,5\}$ and $14\leq n\leq 13+r$.

 (2)-(4). The proofs are  parallel to (1),  we omit the detail.
\end{proof}

\section{Proof of   Theorem \ref{main1}}
%for  $1\leq n\leq 13$ {\color{red}Structural properties of the minimal counterexample\label{structure}}
In this section, we will complete the proof of  Theorem \ref{main1}.
 Assume to the contrary  that
 Theorem \ref{main1} is false, then  there exists a connected graph $G$ of order $n$ such that $|M_G|> f(n),$ which is called a {\it counterexample}.
We say a counterexample $G$ to be {\it minimal} if
for every counterexample $G'$ either $|V(G')|>|V(G)|$ or $|V(G')| = |V(G)|$  and $\tau(G')\geq \tau(G)$. %Consequently, to prove Theorem \ref{main1}, we only need to show that
% \begin{proposition} \label{302}
%There has no counterexample.
%  \end{proposition}

To prove Theorem \ref{main1}, we below focus on establishing several incompatible properties for the minimal counterexample.
Let $G$ be a minimal counterexample and $u$ be any vertex of $G$. Then\\
{\bf Property A.} $1<d(u)<n-1$;\\
{\bf Property B.} if $d(u)\ge 6$, then $u$ is a cutpoint;\\
{\bf Property C.} $G$ does not have the  pendant   blocks  $K_3, K_4, K_5, K_6$;\\
{\bf Property D.} $G$ has no non-cutpoints.

%\begin{itemize}
%\item  $n\ge 15$; \vspace{-4mm}
%\item $1<d(u)<n-1$; \vspace{-4mm}
%\item $u$ is a cutpoint if $d(u)\ge 6$;\vspace{-4mm}
% \item $d(u)\notin \{2,3,4,5\}$ if $u$ is a non-cutpoint.\vspace{-4mm}
%\end{itemize}
We begin with the following result, which shows that the order of each minimal counterexample is at least $ 15$.
Due to Tu et al.,  a sharp upper bound  on the number of maximal dissociation sets in a general graph of order $n$ is given as follows.
\begin{proposition}{\em \cite[Theorem 5.1]{Tu}}\label{01}
Let $n$ be an integer with $n\ge 8$ and let $G$ be a (not necessarily connected)  graph of order $n$.  Then
\begin{align*}
|MD_G|\le
\begin{cases}
10^{t} & ~  {\rm if}~ n=5t; \\
15\cdot 10^{t-1}& ~  {\rm if}~ n=5t+1; \\
225\cdot 10^{t-2}&  ~  {\rm if}~ n=5t+2; \\
36\cdot 10^{t-1}&  ~  {\rm if}~ n=5t+3; \\
6\cdot 10^{t}&  ~  {\rm if}~ n=5t+4, \\
\end{cases}
\end{align*}
where $t$ is a positive integer, and equality holds iff
\begin{align*}
G=
\begin{cases}
tK^*_5 & ~  {\rm if}~ n=5t; \\
K^*_6\cup (t-1)K^*_5 & ~{\rm if}~ n=5t+1; \\
2K^*_6\cup (t-2)K^*_5 &  ~ {\rm if}~ n=5t+2; \\
2K^*_4\cup (t-1)K^*_5&  ~ {\rm if}~ n=5t+3; \\
K^*_4\cup tK^*_5&   ~{\rm if}~ n=5t+4, \\
\end{cases}
\end{align*}
where $K_m^*$ is obtained from the complete graph $K_m$ by possibly deleting $0\le i\le m/2$ disjoint edges.
\end{proposition}

As mentioned earlier, each maximal induced matching is a special kind of dissociation set. Moreover, one finds that, for the complete graph $K_m$,
each maximal induced matching is a maximal dissociation set, and vice versa. Therefore, we have
\begin{theorem}\label{TUGONG}
Let $n$ be an integer with $n\ge 8$ and $G$ be an arbitrary graph  in ${\mathcal G}_n$. Then
\begin{align*}
|M_G|\le q(n):=
\begin{cases}
10^{t} & ~ {\rm if}~ n=5t; \\
15\cdot 10^{t-1}& ~ {\rm if}~ n=5t+1; \\
225\cdot 10^{t-2}&  ~ {\rm if}~ n=5t+2; \\
36\cdot 10^{t-1}&  ~ {\rm if}~ n=5t+3; \\
6\cdot 10^{t}&  ~  {\rm if}~ n=5t+4, \\
\end{cases}
\end{align*}
with equality   iff
\begin{align*}
G=Q_n:=
\begin{cases}
tK_5 & ~ {\rm if}~ n=5t; \\
K_6\cup (t-1)K_5 & ~  {\rm if}~ n=5t+1; \\
2K_6\cup (t-2)K_5 &  ~ {\rm if}~ n=5t+2; \\
2K_4\cup (t-1)K_5&  ~ {\rm if}~ n=5t+3; \\
K_4\cup tK_5&  ~  {\rm if}~ n=5t+4, \\
\end{cases}
\end{align*}where $t$ is a positive integer.
\end{theorem}

Next we borrow the upper bound  $q(n)$ in Theorem \ref{TUGONG} to prove our main Theorem \ref{main1} for
$n\leq 14.$
%\begin{remark}
%  Let $n$ be a positive integer. Then one can verify that, for each $n$, $q(n)\le g(n)$ with equality if $n=0~(mod~5)$,
%  where  $g(n)$ is defined in Eq.~(\ref{(2.1)}). In view of $q(n)$ is a piecewise function on $n$, we in the following sometimes replace   $g(n)$ by $q(n)$ for simplification.
%\end{remark}
%First of all,  we  show that Theorem \ref{main1} holds for $n\le 13$, which implies the property that every counterexample has order at least $14$.
%Denote by  ${\mathcal G}_n$ the set of all connected graphs of order $n$.
\begin{theorem}\label{g13}  Each minimal counterexample has order at least $15$.
%\\
%	{\color{red} Let $n(\le 13)$ be an integer and $G$ be the graph having the maximum number of maximal induced matchings in ${\mathcal G}_n$, then $| M_{G} |=f(n)$, where $f(n)$ is defined as above.}
%
%Let $G\in {\mathcal G}_n$. If $n \le 13$, then $| M_{G} |=f(n)$, where $f(n)$ is defined as above.
\end{theorem}
\begin{proof} %Denote by  ${\mathcal G}_n$  the set of all connected graphs of order $n$ and by ${\mathcal G}^*_n$  the set of all disconnected graphs of order $n$.
By a direct verification, for each $n(\le 7)$, $K_n$ is the unique graph  having the maximum
number of maximal induced matchings in ${\mathcal G}_n$, and  $K_8$ and $K_4*K_4$ are the only two connected graphs having the  maximum number of maximal induced matchings in ${\mathcal G}_8$.
This implies that a minimal counterexample has  order at least $9$.
% Moreover, we find that, for each $n(\le 8)$, $|M_{K_r}|=f(n)$ and $|M_{K_4*K_4}|=f(8)$, where $f(n)$ is defined as Proposition \ref{01}. Thus,

Let $G$ be a connected graph of order  $n(9\le n\le 14)$.
Let $u$ be a vertex with maximum degree of  $G$ and set $d(u)=x$.
Suppose that there are $t$ vertices in $N(u)$, denoted by $v_1,\ldots,v_t$, adjacent to $G-N[u]$. Then $|N[u]\backslash \{v_i|i=1,\ldots,t\}|=x+1-t$. Thus, by Proposition \ref{bopro1},
\begin{equation}\label{a}
  \begin{aligned}
  |M_{G}| & \leq| M_{G-u}| +(x-t)| M_{G-N[u]}|+\sum_{i=1}^t| M_{G-N[u]\cup N[v_i]}|\\
 &= | M_{G-u}| +x| M_{G-N[u]}|-\sum_{i=1}^t(| M_{G-N[u]}|-| M_{G-N[u]\cup N[v_i]}|).
  \end{aligned}%\eqno{(a)}
\end{equation}
%Recall that
% $$ | M_{G} |  \leq | M_{G-u}| +\sum_{p\in N(u)}| M_{G-N[u]\cup N[p]}|\eqno{(a)}$$
%equality holds if each maximal induced matching containing the vertex $u$ is of $\Gamma(u;G)$.

%Note that $|M_{G}|\le q(n-1)$, then we only need to  maximize $\sum_{v\in N(u)}| M_{G-N[u]\cup N[v]}|$. Applying Eq.~(\ref{a}),  each component of $G-N[u]$ has exactly one vertex contained in the set $\{v_1,\ldots,v_t\}$.

We first claim that $x\le n-4$ if $n\ge 9$; $x\le n-5$ if  $n\ge 10$; $x\le n-6$ if $n\ge 12$ and $x\le n-9$ if $n= 14$.
(We only give the proof for $x\le n-4$, since the others are analogous and we omit the detail.)
Otherwise, assume that $x\ge  n-3$, then $|G-N[u]\cup N[v_i]| \le n-(x+1)$ if $N[v_i]\subset N[u]$ and $|G-N[u]\cup N[v_i]| \le n-(x+2)$ otherwise.
Thus $| M_{G} |\le q(n-1)+(x-t)q(n-x-1)+tq(n-x-2)$
 by Ineq.~(\ref{a}).
 However, by a direct verification,  $q(n-1)+(x-t)q(n-x-1)+tq(n-x-2)<f(n)$ holds for each pair of $(n,x)$ with $9\le n\le 14$ and $x\ge n-3$, a contradiction.

 We then claim that $x\ge 3$. Otherwise
  $G$ is a path $P_n$ or a cycle $C_n$, and $|M_G|<f(n)$ obviously.
By those two claims above, if $G$ is a counterexample, then $x$ satisfies: $x\in\{3,4,5\}$ if $n\in \{9,10,14\}$; $x\in\{3,4,5,6\}$ if $n\in \{11,12\}$ and $x\in\{3,4,5,6,7\}$ if $n=13.$
Correspondingly, $|G-N[u]|\in \{3,4,5\}$ if $n=9$;  $|G-N[u]|\in \{4,5,6\}$ if $n=10$; $|G-N[u]|\in \{4,5,6,7\}$ if $n=11$; $|G-N[u]|\in \{5,6,7,8\}$ if $n=12$;
 $|G-N[u]|\in \{5,6,7,8,9\}$ if $n=13$ and $|G-N[u]|\in \{8,9,10\}$ if $n=14$.

For convenience, let $q(n,x)$
 denote the maximum number of maximal induced matchings among all non-regular graphs in ${\mathcal G}_n$ having maximum degree at most $x$.
By a direct calculation, we have $q(3,3)=3$, $q(4,3)=5$, $q(5,3)=7, q(6,3)=8$, $q(7,3)=15$, $q(8,3)=25$, $q(5,4)=9$, $q(6,4)=11$ and $q(7,4)=13$.
If $t\geq 2$, i.e. there are at least two neighbors of $u$ adjacent to $G-N[u]$, then Ineq.~(\ref{a}) yields $|M_G|\leq q(n-1,x)+(x-2)q(n-x-1,x)+2q(n-x-2,x)$.
For $n\notin \{9,14\}$ with $x\neq 4$, it can be directly calculated that $|M_G|\leq f(n)$  for  each pair of $(n,x)$ listed above. For $n\in \{9,14\}$ and $x=4$, we get that that $|M_G|\leq f(n)$ by analyzing the specific graphic structure.
When  $t=1,$  combining the obtained $q(n,x)$ and a careful analysis of the
possible graphic structure, we obtain that  $|M_G|\leq f(n).$
Consequently, there is no minimal counterexample of order less than $15$ and the result follows.
\end{proof}

In addition, we need the following result which shows that each minimal counterexample has no false twins.
\begin{lemma} \label{false}
Let $G$ be a minimal counterexample.  Then $G$ does not have a pair of false twins $u,v$ $($i.e. $N_G(u)=N_G(v)$$)$ with  $|T_G(u)| = |T_G(v)| =1.$
\end{lemma}
\begin{proof}
Suppose not, then we can construct the connected graph $G'$ obtained from $G$ by just connecting $u,v$.
Then $\tau(G')< \tau(G)$ as $N_{G'}[u]=N_{G'}[v]$.
We claim that each maximal induced matching $M\in M_G(\emptyset,\{v\})$  is of $A$-type, i.e. $\beta(v;G)=0$, and thus $|M_G|=|M_{G-v}|+\alpha(v;G)$ by Corollary \ref{inhold0}.
W.l.o.g. let $vp\in M$ where $vp\in E(G)$. Then one can verify that $M\setminus \{vp\}\cup \{up\}$ is an induced matching of  $G-v$.  So $M$  is of $A$-type by the definition.
Combining $G-v= G'-v$ and $G-N[v]\cup N[p] = G'-N[v]\cup N[p]$ where $p\in N(v)$,
we get that  each maximal induced matching $M\in M_G(\emptyset,\{v\})$  is also of $A$-type of $G'$.
Applying Proposition  \ref{bopro1} to $G'$, we have that $|M_{G'}|\geq |M_{G'-v}|+\alpha(v;G')\geq |M_{G-v}|+\alpha(v;G)$. Consequently, $|M_{G'}|\geq|M_G|>f(n)$ and $\tau(G')< \tau(G)$, contradicting the choice of $G$.
\end{proof}

\subsection{Property A: $1<d(u)<n-1$ holds for any vertex $u$ of a minimal counterexample}
%\section{Structural properties of the minimal counterexample\label{structure}}
%In this section, we  prove Theorem \ref{main1}  by contradiction for  $n \ge 14.$ Suppose that the conclusion is false, then there exists a connected graph $G$ of order $n$ such that $|M_G|> f(n),$ which is called a {\it counterexample}.
%We say a counterexample $G$ to be {\it minimal} if
%for every counterexample $G'$ either $|V(G')|>|V(G)|$ or $|V(G')| = |V(G)|$  and $\tau(G')\geq \tau(G)$.
%Then we will characterize a list of structural properties of the minimal counterexample $G$, and finally prove that $G$ does not exist, yielding the desired contradiction.
%
%As a consequence of Theorem \ref{g13}, we have
%\begin{corollary}\label{g2}
%	Let $G$ be a minimal counterexample. Then $n \ge 14.$
%\end{corollary}

Hereafter, based on Theorem \ref{g13}, we always suppose  that each minimal counterexample has order at least $15$.
We begin with the following  proposition.
\begin{proposition}  \label{boproso}
Let $G$ be a minimal counterexample   and let
$H$ be a subgraph of $G$ of order $m.$ Then $|M_H|\leq g(m)$. Moreover, if $H$ is a connected proper induced  subgraph of $G$, then $|M_H|\leq f(m)$.
\end{proposition}
\begin{proof} The first part follows from Lemma \ref{Gupta0}.
For the second part, if $H$ is connected and $m<n$, then $H$ is not a counterexample from the choice of
$G$, which implies that $|M_H|\leq f(m)$.
\end{proof}

%\subsection{\textbf{Properties of vertex degree in  a minimal counterexample}\label{vextexdegree}}
\begin{lemma}   \label{bolem1}
Let $G$ be a minimal counterexample. Then $G$ has no pendant vertices.
\end{lemma}
\begin{proof} For contradiction, assume that there exists a vertex $v$ with $d_{G}(v)=1$. Let $u$ be the unique neighbor of $v$. Let $G'$ be the graph obtained from $G$ by connecting $v$ to each vertex of $N(u).$ Then we can verify that $G'$ is connected and $\tau(G')< \tau(G)$ as $N_{G'}[u]=N_{G'}[v]$.
Note that $G-v= G'-v$ and $G-N[v]\cup N[u] = G'-N[v]\cup N[u]$, then
$|M_{G-v}|=|M_{G'-v}|$ and  each maximal induced matching of $G$ containing the edge $uv$ is of $A$-type if and only if such a maximal induced matching is of $A$-type in  $G'$.
In view of $d_{G}(v)=1$, then $\alpha(v;G)\leq \alpha(v;G')$ and $\varphi(v;G)=\beta(v;G).$ Thus $|M_G|= |M_{G-v}|+\alpha(v;G)$ by Proposition  \ref{bopro1}. Applying Proposition  \ref{bopro1} again to $G'$, we have $|M_{G'-v}|+\alpha(v;G')\leq |M_{G'}|$. Therefore, $|M_{G'}|\geq|M_G|>f(n)$ and $\tau(G')< \tau(G)$, contradicting the choice of $G$.
\end{proof}
 %and.
%
%
%the contribution of the edge $uv$ to $\theta(v;G)$ equals to the contribution of the edge $uv$ to $\gamma(v;G')$.
%Then we construct a new connected And we would prove that
%
%Since
%Due to
%Besides, checking the proof of   Proposition  \ref{bopro1}, we can obtain  that
%when $d_{G}(v)=1$.
%Combining $|M_{G'-v}|+\gamma(v;G')\leq |M_{G'}|$, it follows that $|M_{G'}|\geq |M_G|.$
%Since $u,v$ are twins in $G'$,  it is obvious that $\tau(G')< \tau(G)$, yielding the desired contradiction.

\begin{lemma} \label{dn-1}
Let $G$ be a minimal counterexample.  Then $G$ has no vertices of degree $n-1.$
\end{lemma}
\begin{proof}
For contradiction, assume that $d(v)=n-1.$ By Proposition \ref{bopro1},
$$|M_G|\leq |M_{G-v}|+\sum_{p\in N(v)}|M_{G-N[v]\cup N[p]}|\leq g(n-1)+n-1\leq f(n),$$
which contradicts the fact that $G$ is a counterexample.
\end{proof}

\subsection{Property B:  $u$ is a cutpoint  if $d(u)\ge 6$}
\begin{lemma}   \label{bolem6c}
Let $G$ be a minimal counterexample.
If there is a  vertex  $v\in V(G)$ such that $d(v)\geq 6$, then $v$ is a cutpoint of $G$.
\end{lemma}
\begin{proof}  Suppose for contradiction that $v$ is a non-cutpoint, then $G-v$ is connected and thus $|M_{G-v}|\leq f(n-1)$.
By Lemma \ref{dn-1}, there is
a neighbor $p\in N(v)$ such that $N[p]\nsubseteq N[v]$.
Due to Proposition \ref{bopro1} and Lemma \ref{gong3},  it is easy to verify that when $d(v)\geq 6$ and $n\geq 15$,
\begin{equation*}\label{bocom1}
\begin{aligned}
|M_G|& \leq |M_{G-v}|+\sum_{p\in N(v)}|M_{G-N[v]\cup N[p]}|\\
&\leq f(n-1)+(d(v)-1)g(n-d(v)-1)+g(n-d(v)-2)\\
&\leq f(n-1)+5g(n-7)+g(n-8)\\
&=(10^{-\frac{1}{5}}+5\cdot  10^{-\frac{6}{5}}+10^{-\frac{7}{5}})\cdot 10^{\frac{n-1}{5}}+f_1(n-1)\\
&<0.9860\cdot 10^{\frac{n-1}{5}}+f_1(n-1)<f(n),
\end{aligned}
\end{equation*}
contradicting the assumption that $G$ is a counterexample.  The result follows.
\end{proof}

Combining Lemma   \ref{bolem1} with  Lemma \ref{bolem6c}, we can confine the degree of non-cutpoints in a minimal counterexample $G$ as follows.
\begin{corollary} \label{structureb}
Let $G$ be a minimal counterexample and $v$ be a non-cutpoint of $G$. Then $2\leq d(v)\leq 5$.
\end{corollary}

\subsection{Property C:  each minimal counterexample has no  pendant   blocks  $K_3, K_4, K_5, K_6$}
Let $H$ be a subgraph of  graph $G$. %$v\in V(G)$ and. Recall that $N_G(v)=\{u\in V(G)~|~uv\in E(G)\}$.
When $v\in V(G)$ and  $v\notin V(H)$, we  define $N_H(v)=N_G(v)\cap V(H)$ and $d_H(v)=|N_H(v)|.$ We need establish several lemmas.
\begin{lemma} \label{bo1}
Let $G$ be a minimal counterexample.
Suppose that there exists a pendant block $K_r$ of $G$ with cutpoint $v$. If $4\leq r\leq 6,$ then\\
{\em (1).} $G-K_r$ is disconnected if $d(v)\ge r+1$;\\
{\em (2).} $G-N[v]$ is disconnected if $d(v)= r$.
\end{lemma}
\begin{proof} (1).
By Corollary \ref{bocor2}, $|M_G|={{r-1}\choose 2}\cdot |M_{G-K_r}|+\sum_{p\in N(v)}|M_{G-N[v]\cup N[p]}|.$
%\begin{equation*}\label{bocut4}
%\begin{aligned}
%|M_G|={{r-1}\choose 2}\cdot |M_{G-K_r}|+\sum_{p\in N(v)}|M_{G-N[v]\cup N[p]}|.
%\end{aligned}
%\end{equation*}
For contradiction, assume that $H=G-K_r$ is connected, then $|M_{H}|\leq f(n-r)$.
For $p\in N(v)$,   $|M_{G-N[v]\cup N[p]}|\leq g(n-d(v)-1)$ if $N_{G}[p]\subseteq N_{G}[v]$;
and $|M_{G-N[v]\cup N[p]}|\leq g(n-d(v)-2)$ otherwise.
Lemma \ref{dn-1} implies that there is a neighbor $p\in N_{H}(v)$
such that $N_{G}[p]\nsubseteq N_{G}[v]$.
Consequently,
\begin{equation*}\label{boi1}
\begin{aligned}
|M_G|&\leq {{r-1}\choose 2}\cdot f(n-r)+(d(v)-1)\cdot g(n-d(v)-1)+g(n-d(v)-2)\\
&\leq {{r-1}\choose 2}\cdot f(n-r)+r\cdot g(n-r-2)+g(n-r-3)<f(n),
\end{aligned}
\end{equation*}
where the last two inequalities follow  from Lemmas  \ref{gong3} and \ref{bolemma0011} $(2)$, respectively.
This contradicts the choice of $G$.

(2). The proof is an analogue of (1), so we omit the detail.
(To prove $|M_G|\leq f(n)$, one is referred to  Lemma \ref{bolemma0011} $(1)$.)
\end{proof}

\begin{lemma} \label{ab}
Let $G$ be a minimal counterexample. If $K_r$, $r\in \{4,6\},$ is a pendant block of $G$ with cutpoint $v$, then $d(v)\geq r+1$.
\end{lemma}
\begin{proof}
Assume for contradiction that $G$ contains a pendant block $K_r$ with cutpoint $v$ satisfying $d(v)= r$.
Note that  $|N_{K_r}(v)|=r-1$,
then $|N_{G-K_r}(v)|=1$.
Denote by $N_{G-K_r}(v)=\{u\}$, and  $H=G-K_r-u$.
Lemma \ref{bo1} $(2)$ shows that the induced subgraph $H$  is disconnected.
Let $H_1,H_2,\ldots, H_s$ be all components of $H$ with $|V(H_i)|=n_i$ and $s\geq 2$.
By Corollary \ref{bocor2}, $|M_G|={{r-1}\choose 2}\cdot |M_{G-K_r}|+\sum_{p\in N(v)}|M_{G-N[v]\cup N[p]}|.$
%\begin{equation*}\label{bocut4}
%\begin{aligned}
%|M_G|={{r-1}\choose 2}\cdot |M_{G-K_r}|+\sum_{p\in N(v)}|M_{G-N[v]\cup N[p]}|.
%\end{aligned}
%\end{equation*}
Observe that $G-K_r$ is connected, then $|M_{G-K_r}|\leq f(n-r).$
For every $p\in  N_{K_r}(v)$, it holds that
\begin{equation}\label{boc4}
\begin{aligned}
|M_{G-N[v]\cup N[p]}|= |M_H|=\prod_{i=1}^s |M_{H_{i}}| \leq \prod_{i=1}^s f(n_i).
\end{aligned}
\end{equation} When $p=u$,
\begin{equation}\label{boc6}
\begin{aligned}
|M_{G-N[v]\cup N[u]}|\leq \prod_{i=1}^s |M_{H_{i}-N_{H_{i}}(u)}|\leq \prod_{i=1}^s g(n_i-1)=10^{\frac{n-r-1-s}{5}}\leq 10^{\frac{n-r-3}{5}},
\end{aligned}
\end{equation}as $|H_{i}-N_{H_{i}}(u)|\leq n_i-1,$ $\sum_{i=1}^s n_i=n-r-1$  and $s\geq 2.$
Consequently, we have that
\begin{equation}\label{ad}
\begin{aligned}
|M_G|&\leq {{r-1}\choose 2}\cdot f(n-r)+(r-1)\cdot \prod_{i=1}^s f(n_i)+ 10^{\frac{n-r-3}{5}}.
\end{aligned}
\end{equation}
%$$|M_G|\leq \frac{(r-1)(r-2)}{2}f(n-r)+(r-1)\prod_{i=1}^s f(n_i)+ 10^{\frac{n-r-3}{5}}.\eqno{(3)}$$
If there exists a  component $H_j\neq K_5,$ then  $f(n_j)\leq \frac{f(4)}{g(4)}\cdot g(n_j)$ by Lemma \ref{botool1}.
Then $\prod_{i=1}^s f(n_i)\leq \frac{f(4)}{g(4)}\cdot \prod_{i=1}^s g(n_i)= \frac{f(4)}{g(4)}\cdot 10^{\frac{n-r-1}{5}}$, as $f(n_i)\leq g(n_i)~(i\neq j)$ and $\sum_{i=1}^s n_i= n-r-1$.
Further, Ineq.~(\ref{ad}) yields
\begin{equation*}\label{ae}
\begin{aligned}
|M_G|&\leq {{r-1}\choose 2}\cdot f(n-r)+ \frac{f(4)}{g(4)}\cdot (r-1)\cdot 10^{\frac{n-r-1}{5}}+ 10^{\frac{n-r-3}{5}}< f(n),
\end{aligned}
\end{equation*}
where the last inequality follows from Lemma \ref{bolemma0011} $(4)$,
contradicting the fact that $G$ is a counterexample. This implies that $H$ is the disjoint union of $K_5$.

When $H_i=K_5~(1\leq i\leq s)$, Ineq.s~(\ref{boc4}) and (\ref{boc6})   yield
\begin{equation} \label{boc465}
\begin{aligned}
|M_G|&\leq {{r-1}\choose 2}\cdot f(n-r)+(r-1)\cdot \prod_{i=1}^s |M_{H_{i}}|+\prod_{i=1}^s |M_{H_{i}-N_{H_{i}}(u)}|\\
&\leq {{r-1}\choose 2}\cdot f(n-r)+(r-1)\cdot 10^{\frac{n-r-1}{5}}+6^{\frac{n-r-1}{5}}<f(n),
\end{aligned}
\end{equation}
where the last  two inequalities follow from $|M_{H_{i}}|=|M_{K_5}|=10, |M_{H_{i}-N_{H_{i}}(u)}|\leq |M_{K_4}|=6~(1\leq i\leq s)$ and Lemma \ref{bolemma0011}~$(3)$, respectively. This contradicts the fact that $G$ is a counterexample. The result follows.
\end{proof}

%{\color{red}Recall that  $G$ is the  minimum counterexample and let $K_r$ be a pendant block of $G$ with cutpoint $u$. Denote by $N(u;G-K_r)=\{v\}$, and  $H=G-K_r-v$.}
\begin{lemma} \label{bopen5}
Let $G$ be a minimal counterexample.  If $K_5$ is a pendant block of $G$ with cutpoint $v$,  then $d(v)\geq 6$.
\end{lemma}
\begin{proof}
For contradiction, suppose that $G$ contains a pendant block $K_5$  with cutpoint $v$ satisfying $d(v)= 5$.
Since $|N_{K_5}(v)|=4$,   we may denote by $N_{G-K_5}(v)=\{u\}$  and  $H=G-K_5-u$.
Lemma \ref{bo1} $(2)$ shows that the induced subgraph $H$  is disconnected.
Let $H_1,H_2,\ldots, H_s$ be all components of $H$ with $|V(H_i)|=n_i$, where $\sum_{i=1}^s n_i=n-6$ and $s\geq 2$.
Similar to the proof of Lemma \ref{ab},
by Ineq. (\ref{boc465}), note that here $r= 5$,
\begin{equation}\label{bocpen56}
\begin{aligned}
|M_G|&\leq {{r-1}\choose 2}\cdot f(n-r)+(r-1)\cdot \prod_{i=1}^s |M_{H_{i}}|+\prod_{i=1}^s |M_{H_{i}-N_{H_{i}}(u)}|\\
&=6f(n-5)+4\cdot \prod_{i=1}^s |M_{H_{i}}|+\prod_{i=1}^s |M_{H_{i}-N_{H_{i}}(u)}|.
\end{aligned}
\end{equation}

\noindent{\bf Claim 1.} {\it  There is no component other than $\{K_4,K_5,K_6\}$ in $H.$}

If there exists some component $H_j\notin  \{K_4,K_5,K_6\}$, then  $|M_{H_j}|\leq f(n_j)\leq 0.9\cdot g(n_j)$ by Lemma \ref{botool1}.
So,
$$\prod_{i=1}^s |M_{H_{i}}|\leq \prod_{i=1}^s f(n_i)\leq 0.9\cdot \prod_{i=1}^s g(n_i)= 0.9\cdot 10^{\frac{n-6}{5}},$$
$$\prod_{i=1}^s |M_{H_{i}-N_{H_{i}}(u)}|\leq \prod_{i=1}^s g(n_i-1)=10^{\frac{n-6-s}{5}}\leq 10^{\frac{n-8}{5}},$$
as $f(n_i)\leq g(n_i)~(i\neq j)$ and $\sum_{i=1}^s n_i= n-6$.
Consequently, when $n\geq 19$, Ineq.~(\ref{bocpen56}) yields
\begin{equation}\label{bocpen51}
\begin{aligned}
|M_G|&\leq 6f(n-5)+4\cdot \prod_{i=1}^s |M_{H_{i}}|+\prod_{i=1}^s |M_{H_{i}-N_{H_{i}}(u)}|\\
&\leq 6f(n-5)+4\cdot 0.9\cdot 10^{\frac{n-6}{5}}+10^{\frac{n-8}{5}}\\
&=(6\cdot 10^{-\frac{5}{5}}+4\cdot 0.9\cdot 10^{-\frac{5}{5}}+10^{-\frac{7}{5}})\cdot 10^{\frac{n-1}{5}}+6f_1(n-5)\\
&<0.9998\cdot 10^{\frac{n-1}{5}}+f_1(n)< f(n),
\end{aligned}
\end{equation}
where the third inequality follows from   $6f_1(n-5)\leq f_1(n)$.
For $15\leq n\leq 18,$ $|M_G|\leq f(n)$ by  a direct verification.
This completes the proof of Claim 1.
%Summarizing,
%we obtain a contradiction to the assumption that $|M_G|> f(n)$ when $n\geq 14$, which completes the proof of Claim 1.

\noindent{\bf Claim 2.} {\it There is no $K_4$ as a component in $H$.}

Assume, for contradiction, that $H$  contains $K_4$ as a  component.
W.l.o.g., let $H_1= K_4$. Then Lemma \ref{ab} forces that $|M_{H_{1}-N_{H_{1}}(u)}|\leq |M_{K_2}|=1.$ Further,
$$\prod_{i=1}^s |M_{H_{i}}|\leq |M_{H_{1}}|\cdot\prod_{i=2}^s f(n_i)\leq  {{4}\choose 2}\cdot\prod_{i=2}^s g(n_i)= 6\cdot 10^{\frac{n-10}{5}},$$
$$\prod_{i=1}^s |M_{H_{i}-N_{H_{i}}(u)}|\leq |M_{H_{1}-N_{H_{1}}(u)}|\cdot\prod_{i=2}^s g(n_i-1)=1\cdot 10^{\frac{n-10-(s-1)}{5}}\leq  10^{\frac{n-11}{5}},$$
as $f(n_i)\leq g(n_i)~(i\geq 2)$, $\sum_{i=2}^s n_i= n-10$  and $s\geq 2$.
So, when $n\geq 19$, Ineq.~(\ref{bocpen56}) yields
\begin{equation*}%\label{bocpen51}
\begin{aligned}
|M_G|&\leq 6f(n-5)+4\cdot \prod_{i=1}^s |M_{H_{i}}|+\prod_{i=1}^s |M_{H_{i}-N_{H_{i}}(u)}|\\
&\leq 6f(n-5)+4\cdot 6\cdot 10^{\frac{n-10}{5}}+10^{\frac{n-11}{5}}< f(n),
\end{aligned}
\end{equation*}
where the last inequality follows by the fact that the upper bound in this case
is strictly smaller than the  upper bound of  Ineq.~(\ref{bocpen51}) (i.e., $6f(n-5)+4\cdot 0.9\cdot 10^{\frac{n-6}{5}}+10^{\frac{n-8}{5}}$).
For the same reason, $|M_G|\leq f(n)$ for $15\leq n\leq 18.$
This completes the proof of Claim 2.

\noindent{\bf Claim 3.} {\it There is no $K_6$ as a component in $H$.}

For contradiction, if   $H$ has exactly two components and one of them is  $K_6$.
By Claims 1 and 2, the other component is $K_5$ or $K_6$.   By Corollary \ref{bocor2}, it can be easily checked  that $|M_G|\leq f(n)$.

Now consider the case that $H$ has at least three  components (i.e. $s\geq 3$) and one of them is  $K_6$.
W.l.o.g., let $H_1= K_6$. Then Lemma \ref{ab} forces that $|M_{H_{1}-N_{H_{1}}(u)}|\leq |M_{K_4}|=6.$ Further,
$$\prod_{i=1}^s |M_{H_{i}}|\leq |M_{H_{1}}|\cdot  \prod_{i=2}^s f(n_i)\leq  {{6}\choose 2}\cdot\prod_{i=2}^s g(n_i)= 15\cdot 10^{\frac{n-12}{5}},$$
$$\prod_{i=1}^s |M_{H_{i}-N_{H_{i}}(u)}|\leq |M_{H_{1}-N_{H_{1}}(u)}|\cdot \prod_{i=2}^s g(n_i-1)=6\cdot 10^{\frac{n-12-(s-1)}{5}}\leq 6\cdot 10^{\frac{n-14}{5}},$$
as $f(n_i)\leq g(n_i)~(i\geq 2)$, $s\geq 3$ and $\sum_{i=2}^s n_i= n-12$.
Consequently, when  $n\geq 19$, Ineq.~(\ref{bocpen56}) yields
$$|M_G|\leq 6f(n-5)+4\cdot 15\cdot 10^{\frac{n-12}{5}}+6\cdot 10^{\frac{n-14}{5}}<f(n),$$
where the last inequality follows by the fact that the upper bound in this case
is strictly smaller than the  upper bound of  Ineq.~(\ref{bocpen51}) (i.e., $6f(n-5)+4\cdot 0.9\cdot 10^{\frac{n-6}{5}}+10^{\frac{n-8}{5}}$).
And for the same reason, $|M_G|\leq f(n)$ for $15\leq n\leq 18.$
This completes the proof of Claim 3.

\noindent{\bf Claim 4.} {\it $|M_G|\leq f(n)$.}

Combining Claims $1,2,3$, we  conclude that every component $H_i=K_5~(1\leq i\leq s)$.
Then when $n\geq 36$, Ineq.~(\ref{bocpen56}) yields
\begin{equation*}%\label{bocpen51}
\begin{aligned}
|M_G|&\leq 6f(n-5)+4\cdot \prod_{i=1}^s |M_{H_{i}}|+\prod_{i=1}^s |M_{H_{i}-N_{H_{i}}(u)}|\\
&\leq 6f(n-5)+4\cdot 10^{\frac{n-6}{5}}+6^{\frac{n-6}{5}}=f(n),
\end{aligned}
\end{equation*}
as $|M_{H_{i}-N_{H_{i}}(u)}|\leq |M_{K_4}|=6$ for every $1\leq i\leq s.$
For  $15\leq n\leq 35$, combining $|M_G|={{r-1}\choose 2}\cdot |M_{G-K_r}|+\sum_{p\in N(v)}|M_{G-N[v]\cup N[p]}|$  with  $H_i=K_5~(1\leq i\leq s)$, it can be easily checked that $|M_G|\leq f(n)$.
By Claim $4$, the result follows.
\end{proof}

%The following two lemmas characterize the  substructure of pendant blocks with smaller order.
\begin{lemma}   \label{gcut5}
Let $G$ be a minimal counterexample. Suppose that $v$ is a cutpoint of $G$ and that
$G_1$ is a component of $G-v$.
If  $|G_1+v|\leq 5$, then $G_1+v$ is a complete graph.	
\end{lemma}
\begin{proof} %Assume, for contradiction, that $G_1+v$ is not a complete graph. Then
%construct a new connected graph $G'$ obtained from $G$ by just transforming $G_1+v$ into a complete subgraph. We will claim that $|M_{G'}|\geq|M_G|$ and $\tau(G')< \tau(G)$, contradicting the choice of $G$.
Let $G_2=G-G_1-v$.
Set $|G_1+v|=r~(\leq 5)$.  Since $G$ contains no pendant vertices by Lemma \ref{bolem1},   $3\le r\le 5$.
Note that $v$ is a cutpoint, then by Proposition \ref{bopro1},
\begin{equation}\label{gbocom25}
\begin{aligned}
|M_G|& \leq |M_{G-v}|+\sum_{p\in N(v)}|M_{G-N[v]\cup N[p]}|\\
&\leq |M_{G_1}|\cdot |M_{G_2}|+|M_{G_1+v}(\emptyset,\{v\})|\cdot |M_{G_2-N_{G_2}(v)}|+|M_{G_2+v}(\emptyset,\{v\})|\cdot |M_{G_1-N_{G_1}(v)}|.
\end{aligned}%\eqno{(3.1)}
\end{equation}
It is easy to check that $|M_{G_1+v}(\emptyset,\{v\})|=d_{G_1}(v)$  and $|M_{G_1-N_{G_1}(v)}|=1$ %(here we emphasize that $|M_{G_1-N_{G_1}(v)}|=1$ if $G_1-N_{G_1}(v)\in \{\emptyset, P_1,2P_1\}$ in Ineq.~(\ref{gbocom25}) for convenience)
if either $r\le 4$ or $r=5$
and $d_{G_1}(v)\ge 2$.
Then, by Ineq. (\ref{gbocom25}), to maximize $|M_G|$ and minimize $\tau(G)$,  $G_1+v=K_r$. Otherwise, if $G_1+v\neq K_r$ in $G$, then we can construct the (connected) graph $G'$ obtained from $G$ by just transforming $G_1+v$ into a complete subgraph.
Since $|M_{G'}|>|M_G|$ and $\tau(G')< \tau(G)$, contradicting the choice of $G$.

It remains to exclude the leaving case where  $r=5$
and $d_{G_1}(v)=1$. For $r=5$ and $d_{G_1}(v)=1$, let $N_{G_1}(v)=w.$ Note that now $w$ is a cutpoint of $G$. Then applying the similar discussion to $w$ as above, to maximize
$|M_G|$ and minimize $\tau(G)$,  $G_1=K_4$. However, if $G_1=K_4$ and $d_{G_1}(v)=1$, then a contradiction to Lemma \ref{ab}
will be yielded. This completes the proof.
\end{proof}
%A {\it coalescence} $G\cdot H$ of two graphs $G$ and $H$ is obtained from $G\cup H$ by identifying a vertex
%$u$ of $G$ with a vertex $v$ of $H$.
%\begin{lemma} (Yuan) \label{bocut9}
%Let $G$ be a minimal counterexample. Suppose that $v$ is a cutpoint of $G$ and that
%$G_1$ is a component of $G-v$.
%If  $|G_1+v|=6$, then either $G_1+v=K_6$ or $G_1=K_5$ and $d_{G_1}(v)=1$.	
%\end{lemma}
Applying all those lemmas above, we below present  Property C step by step.
\begin{lemma}\label{gpb3}
Let $G$ be a minimal counterexample.  Then $G$ does not have a pendant block $K_3$.
\end{lemma}
\begin{proof}  Assume to the contrary that $G[\{u,v,w\}]=K_3$ with $d(u)=d(v)=2$ and $d(w)\ge 3$.
If $d(w)=3,$ then let $N_{G-K_3}(w)=\{x\}.$ Clearly,  $x$ is a cutpoint. Then $G[\{u,v,w,x\}]= K_4$ by Lemma \ref{gcut5},   a contradiction. So we have $d(w)\geq 4.$
Lemma \ref{dn-1} forces that there is a neighbor $p\in N(w)$ such that
$N[p]\backslash N[w] \neq \emptyset$.
Consequently,
\begin{align*}
|M_G|&= |M_G(\{w\},\emptyset)|+ |M_G(\emptyset,\{w\})|\\& \leq g(n-3)+(d(w)-1)g(n-1-d(w))+g(n-2-d(w))\\
& \leq g(n-3)+3g(n-5)+g(n-6)\\
& < 0.9736\cdot 10^\frac{n-1}{5} < f(n),
\end{align*}
where the second inequality follows from Lemma \ref{gong3}, a contradiction. The result follows.
\end{proof}

\begin{lemma}   \label{gcut6}
Let $G$ be a minimal counterexample. Suppose that $v$ is a cutpoint of $G$ and that
$G_1$ is a component of $G-v$.
If  $|G_1+v|=6$, then $G_1+v= K_6$. %either $G_1+v= K_6$. %or $G_1+v= K_1\ast K_5$ and $d_{G_1}(v)=1$.	
\end{lemma}
\begin{proof}  We divide the proof into two cases.

\noindent{\bf Case 1.} {\it $d_{G_1}(v)=1$.}

Let $N_{G_1}(v)=\{w\}$. Then  $w$ is also a cutpoint of $G$. Applying Lemma \ref{gcut5} to $w$,   either $G_1=K_5$ or $G_1$ is the graph of order $5$ obtained by identifying one vertex of two triangles  where $d_{G_1}(w)=4$, as $|G_1|=5$.
This contradicts  Lemma \ref{bopen5} and Lemma \ref{gpb3}, respectively.

\noindent{\bf Case 2.} {\it $d_{G_1}(v)\ge 2$.}

When $|G_1+v|=6$ and $d_{G_1}(v)\ge 3$,  it is easy to see that $|M_{G_1}(\emptyset,\{v\})|=d_{G_1}(v)$ and $|M_{G_1-N_{G_1}(v)}|=1$.
%(here we emphasize that $|M_{G_1-N_{G_1}(v)}|=1$ if $G_1-N_{G_1}(v)\in \{\emptyset, P_1,2P_1\}$  for convenience)
Then, applying Corollary \ref{bocor2} and Ineq. (\ref{gbocom25}), to maximize $|M_G|$ and minimize $\tau(G)$, $G_1=K_{5}$ and $d_{G_1}(v)= 5.$
Consequently, $G_1+v=K_6$. Otherwise, if $G_1+v\neq K_6$ in $G$, then we can construct the (connected) graph $G'$ obtained from $G$ by just transforming $G_1+v$ into $K_6$.
Since $|M_{G'}|>|M_G|$ and $\tau(G')< \tau(G)$, contradicting the choice of $G$.

To prove this lemma, it suffices to exclude the case that  $d_{G_1}(v)=2$.
When $d_{G_1}(v)=2$, set $N_{G_1}(v)=\{u,w\}.$ %we will prove that
Observe that $|M_{G_1-N_{G_1}(v)}|\le 3$  equality holds iff $G_1-N_{G_1}(v)= K_3$.
Similar to the discussion above,  to maximize $|M_G|$ and minimize $\tau(G)$, $G_1= K_5$.
Let $x\in V(G_1)\setminus \{u,w\}.$    One can verify that the edge $ux$ satisfies the requirement of Lemma  \ref{y1}. Thus  we obtain that either
$|M_{G_{u\rightarrow x}}|\geq|M_G|$ and $\tau(G_{u\rightarrow x})=\tau(G)$ or $|M_{G_{x\rightarrow u}}|\geq|M_G|$ and $\tau(G_{x\rightarrow u})=\tau(G)$. This implies that either  $G_{u\rightarrow x}$ or
$G_{x\rightarrow u}$ is also a minimal counterexample.
For $G_{u\rightarrow x}$, there is a pendant block $K_5$ with cutpoint $w$ satisfying $d_G(w)=5$, contradicting Lemma \ref{bopen5}. For $G_{x\rightarrow u}$, note that $d_{G_1}(v)=3$ and thus $G_1+v=K_6$ by the discussion above, a contradiction.  Both contradictions imply that $d_{G_1}(v)\ne 2$.
The result follows.
\end{proof}

%\subsection{\textbf{Forbidden structures}\label{forbiddensubgraph}}
%%If $G$ contains an induced subgraph $H$ such that   $|M_G|\leq f(n)$, then $H$ is called a {\it forbidden subgraph} of $G$.
%In the sequel we shall characterize the  forbidden substructures to obtain the
%structural properties of $G$.

%A pendant block of   graph $G$ is called {\it pendant   block} if it is a complete subgraph.
\begin{lemma} \label{bopen06}
Let $G$ be a minimal counterexample.  Then $G$ does not have a pendant   block $K_6$.
\end{lemma}
\begin{proof}
For contradiction, let $K_6$ be a pendant   block  with cutpoint $v$.
By Lemma \ref{ab}, it suffices to exclude the case that $d(v)\geq 7$.
By Lemma \ref{bo1} $(1)$,  $G-K_6$ is disconnected. Let $H_1,H_2,\ldots, H_s$ be all components of  $H=G-K_6$ with $|V(H_i)|=n_i$ and $s\geq 2$. To find an upper bound on $|M_G|$, we first  confine the components of $H$ as follows.

\noindent{\bf Claim 1.} {\it There is neither  $K_6$ nor $K_6-e$ as a component in $H.$}

If not, assume that $H_1= K_6$. Lemma \ref{ab} forces that $v$ is adjacent to at least two vertices of $H_1$, then there are non-cutpoints of degree $6$ in $H_1$, which  contradicts
Lemma \ref{bolem6c}.

Now assume that  $H_1= K_6-e$, where $e=xy$. If $v$ is adjacent to exactly one vertex, say $h_1$, of $H_1$, then $h_1$ is also a cutpoint of  $G$. Lemma \ref{gcut6} shows that $H_1= K_6,$ a contradiction.
If $v$ is adjacent to at least two vertices of $H_1$, in order to exclude non-cutpoints of degree $6$, then $N_{H_1}(v)=\{x, y\}.$ Then $x$ and $y$ become false twins, which  contradicts
Lemma \ref{false}.

\noindent {\bf Claim 2.} {\it There is no component of order $5$ in $H$.}

For contradiction, w.l.o.g., let  $|H_1|=5,$ then  $v+H_1=K_6$ by  Lemma \ref{gcut6}.
Further, if there is another component, say  $H_i~(i\neq 1)$, such that  $n_i\notin  \{4, 5\}$, then $f(n_i)\leq 0.85g(n_i)$ combining Claim 1 with  Lemma \ref{botool1}.
Let $w(\neq v)$ be a vertex of the pendant block $K_6$ in the condition of  Lemma   \ref{bopen06}.
%Let $w(\neq v)$ be a vertex of the  required pendant block $K_6$.
%Denote by $w(\neq v)$ a vertex of the  required pendant block $K_6$.
Now consider the upper bound on $|M_{G}|$ by
\begin{equation} \label{bopen6}
\begin{aligned}
|M_{G}|\leq |M_{G-w}|+\sum_{p\in N(w)}|M_{G-N[w]\cup N[p]}|.
\end{aligned}
\end{equation}
Observe that $|M_{G-w}|\leq f(n-1)$ as $G-w$ is connected. When $p\in N_{K_6}(w)\setminus \{v\}$, it is easy to verify that
$|M_{G-N[w]\cup N[p]}|\leq 0.85\cdot 10^{\frac{n-6}{5}}$.
When $p=v$, $|M_{G-N[w]\cup N[v]}|\leq 10^{\frac{n_2+\cdots +n_s-(s-1)}{5}}\leq 10^{\frac{n-12}{5}}$ as $n_2+\cdots +n_s=n-11$ and $s\geq 2$.
When $n\geq 15$, combining the obtained upper bounds yields
\begin{equation*}%\label{bocpen51}
\begin{aligned}
|M_G|&\leq f(n-1)+4\cdot 0.85\cdot 10^{\frac{n-6}{5}}+10^{\frac{n-12}{5}}\\
&=(10^{-\frac{1}{5}}+4\cdot 0.85\cdot 10^{-\frac{5}{5}}+10^{-\frac{11}{5}})\cdot 10^{\frac{n-1}{5}}+ f_1(n-1)\\
&<0.9773\cdot 10^{\frac{n-1}{5}}+f_1(n-1)<f(n),
\end{aligned}
\end{equation*}
a contradiction. %It is easy to check that $|M_G|\leq f(n)$ for $n=14.$
Now we  conclude that  every component  $H_i~(i\neq 1)$ satisfies $n_i\in  \{4, 5\}.$
Recall that  $v$ is a cutpoint of $G$, then $v+H_i= K_{n_i+1}~(i\neq 1)$ due to Lemmas \ref{gcut5} and \ref{gcut6}.
Combining  $v+H_1=K_6$, we obtain that
$v$ is adjacent to every vertex of $G$, which contradicts Lemma \ref{dn-1}. Claim $2$ follows.

\noindent {\bf Claim 3.} {\it There is no component of order $4$ in $H$.}

Suppose not, let $n_1=4,$ then  $v+H_1= K_5$  by  Lemma \ref{gcut5}.
Further, if there is another component, say  $H_i~(i\neq 1)$, such that
$n_i\neq 4$, then $f(n_i)\leq 0.85g(n_i)$ combining Claims $1,2$ with Lemma \ref{botool1}.
Then Ineq. (\ref{bopen6}) yields
\begin{equation*}
\begin{aligned}
|M_G|& \leq f(n-1)+4\cdot6\cdot 0.85\cdot 10^{\frac{n-10}{5}}+10^{\frac{n-11}{5}}\\
&< 0.9643\cdot 10^{\frac{n-1}{5}}+f_1(n-1)<f(n).
\end{aligned}
\end{equation*}
From this, we  conclude that the order of any component $H_i~(i\neq 1)$ is 4 and
$v+H_i= K_5$ by  Lemma \ref{gcut5}. Note that the degree of cutpoint $v$ is $n-1$, which contradicts Lemma \ref{dn-1}. Claim $3$ follows.

From Claims $1,2,3$ and Lemma \ref{botool1}, it follows that for any  $1\leq i\leq s,$ $f(n_i)\leq 0.85g(n_i)$.
Now consider the upper bound on Ineq. (\ref{bopen6}), for any $p\in N_{K_6}(w)\setminus \{v\}$,
$|M_{G-N[w]\cup N[p]}|\leq 0.85^{s}\cdot 10^{\frac{n-6}{5}}$. When $p=v$,
$|M_{G-N[w]\cup N[p]}|\leq \prod_{i=1}^s g(n_i-1)=10^{\frac{n-6-s}{5}}$.
Putting all this together, we have
\begin{equation*}
\begin{aligned}
|M_G| \leq f(n-1)+4\cdot  0.85^{s}\cdot 10^{\frac{n-6}{5}}+10^{\frac{n-6-s}{5}}
< 0.9598\cdot 10^{\frac{n-1}{5}}+f_1(n-1)<f(n),
\end{aligned}
\end{equation*}
 a contradiction. The result follows.
\end{proof}

\begin{lemma} \label{bopen05}
Let $G$ be a minimal counterexample.  Then $G$ does not have a pendant   block  $K_5$.
\end{lemma}
\begin{proof}
Similar to the proof of Lemma  \ref{bopen06}.
For contradiction, let $K_5$ be a pendant   block  with cutpoint $v$.
By Lemma \ref{bopen5}, it suffices to exclude the case that $d(v)\geq 6$.
By Lemma \ref{bo1} $(1)$, $G-K_5$ is disconnected. Let $H_1,H_2,\ldots, H_s$ be all components of  $H=G-K_5$ with $|V(H_i)|=n_i$ and $s\geq 2$.

\noindent {\bf Claim 1.} {\it There is neither  $K_6$ nor $K_6-e$ as a component in $H.$}
The proof is the same as the  Claim 1 of Lemma \ref{bopen06}.

\noindent {\bf Claim 2.} {\it There is no component of order $5$ in $H.$}

If not, let $|H_i|=5.$ Note that $v$ is a cutpoint of $G$, then
 $v+H_i=K_6$ by Lemma \ref{gcut6}, which  contradicts Lemma \ref{bopen06}.

\noindent {\bf Claim 3.} {\it There is no  $K_4$ as a component in $H.$}

For contradiction, assume that $H_1= K_4.$ Then $v+H_1= K_5$  by  Lemma \ref{gcut5}.   We claim that there is no more $K_4$ in $H$. Otherwise, suppose $H_2= K_4,$ then $v+H_2= K_5$. Now we can construct a new connected graph  $G'$ obtained from $G$ by just transforming the three pendant $K_5$ into two pendant $K_7$. It is easy to verify that $|M_{G'}|\geq|M_G|$ and $\tau(G')< \tau(G)$, contradicting the choice of $G$.
Now combining   Claims  $1,2$ with Lemma \ref{botool1}, we conclude that for any component $H_i\neq H_1$, $f(n_i)\leq 0.85g(n_i)$.

Let $u\in  N_{K_5}(v)$. Now consider $|M_G|=|M_{G-u}|+\sum_{p\in N(u)}|M_{G-N[u]\cup N[p]}|.$
Clearly, $|M_{G-u}|\leq f(n-1)$. When $p\neq v,$ $|M_{G-N[u]\cup N[p]}|\leq 6\cdot 0.85^{s-1}\cdot g(n-9)$.
When $p=v,$ $|M_{G-N[u]\cup N[p]}|\leq \prod_{i=2}^s g(n_i-1)=10^{\frac{n-8-s}{5}}$.
Putting all this together, we have
\begin{equation*}
\begin{aligned}
|M_G|\leq f(n-1)+3\cdot 6\cdot 0.85^{s-1}\cdot 10^{\frac{n-9}{5}}+10^{\frac{n-8-s}{5}}.
\end{aligned}
\end{equation*}
 When $s\geq 3,$
 \begin{equation*}
\begin{aligned}
|M_G|&\leq  (10^{-\frac{1}{5}}+18\cdot 0.85^{2}\cdot 10^{-\frac{8}{5}}+10^{-\frac{10}{5}})\cdot 10^{\frac{n-1}{5}} +f_1(n-1)\\
&< 0.9677\cdot 10^{\frac{n-1}{5}}+f_1(n-1)<f(n),
\end{aligned}
\end{equation*}
 a contradiction.
Next we exclude $s=2.$  Denote by  $|N_{H_2}(v)|=d~(d\geq 1)$.
Recall that $v+H_1= K_5$ and $|M_G|=|M_{G-v}|+\sum_{p\in N(v)}|M_{G-N[v]\cup N[p]}|$ by Corollary \ref{bocor2}.
It is easy to see that $|M_{G-v}|\leq 6\cdot 6\cdot 0.85g(n_2)$.
If $p\in N_{K_5}(v),$ $|M_{G-N[v]\cup N[p]}|\leq g(n_2-d)$.
If $p\in N_{H_2}(v),$ $\sum_{p\in N_{H_2}(v)}|M_{G-N[v]\cup N[p]}|\leq (d-1)\cdot g(n_2-d)+g(n_2-d-1)$.
Hence we find that
\begin{equation*}%\label{bocpen51}
\begin{aligned}
|M_G|&\leq 6\cdot 6\cdot 0.85g(n_2)+8\cdot g(n_2-d)+(d-1)\cdot g(n_2-d)+g(n_2-d-1)\\
&=30.6\cdot 10^{\frac{n-9}{5}}+(d+7+10^{-\frac{1}{5}})\cdot 10^{\frac{n-9-d}{5}}\\
&\leq (30.6\cdot 10^{-\frac{8}{5}}+(8+10^{-\frac{1}{5}})\cdot 10^{-\frac{9}{5}})\cdot 10^{\frac{n-1}{5}}\\
&\leq 0.9055\cdot 10^{\frac{n-1}{5}}<10^{\frac{n-1}{5}}+f_1(n)=f(n),
\end{aligned}
\end{equation*}
which   contradicts   the fact that $G$ is a  counterexample. Claim $3$ follows.

Combining Claims $1, 2, 3$ with Lemma \ref{botool1}, we can get that for any component  $H_i$ of $H,$
$|M_{H_i}|\leq f(n_i)\leq 0.85g(n_i).$ Further, we have the following claim.

\noindent {\bf Claim 4.} {\it For any $1\leq i\leq s$, $|M_{H_i-N_{H_i}(v)}|\leq 0.85g(n_i-1).$}

If $|N_{H_i}(v)|\geq 2$, then $|M_{H_i-N_{H_i}(v)}|\leq g(n_i-2)<0.85g(n_i-1).$
If $|N_{H_i}(v)|=1$, denoted by $N_{H_i}(v)=w_i$, then  $w_i$ is a cutpoint of $G$.
In the subgraph $H_i-N_{H_i}(v)$, similar to the  proof of Claims $1, 2$, there is no $K_6$, $K_6-e$, subgraph of order $5$ as a component. If there is a $K_4$ in $H_i-N_{H_i}(v)$, then $w_i+K_4$ is also a pendant   block $K_5$ in  $G$.
Then Claim  $2$ forces that
there is another component, denoted by $H^{1}_i, H^{2}_i,\ldots,H^{s_i}_i$,  in $H_i-N_{H_i}(v)$.
Applying the above conclusion to $w_i+K_4=K_5$, we have that $|M_{H^{j}_i}|\leq 0.85g(|H^{j}_i|)$ for any $1\leq j\leq s_i$.
So $|M_{H_i-N_{H_i}(v)}|=|M_{K_4}|\cdot \prod_{j=1}^{s_i} |M_{H^{j}_i}|\leq 6\cdot 0.85g(n_i-5)<0.85g(n_i-1).$
Claim $4$ follows.

Recall that $u\in  N_{K_5}(v)$. Now we can have a better upper bound on
$$|M_G|=|M_{G-u}|+\sum_{p\in N(u)}|M_{G-N[u]\cup N[p]}|.$$
Clearly, $|M_{G-u}|\leq  f(n-1).$ If $p\neq v$, $|M_{G-N[u]\cup N[p]}|\leq 0.85^{s}\cdot g(n-5).$ If $p=v$, $|M_{G-N[u]\cup N[p]}|=\prod_{i=1}^{s} |M_{H_i-N_{H_i}(v)}|\leq 0.85^{s}g(n-5-s).$
Consequently, we obtain
\begin{equation*}
\begin{aligned}
|M_G|\leq f(n-1)+ 3\cdot 0.85^{s}\cdot g(n-5)+0.85^{s}g(n-5-s).
\end{aligned}
\end{equation*}
If $s\geq 3$, it is easy to verify that
\begin{equation*}
\begin{aligned}
|M_G|&\leq (10^{-\frac{1}{5}}+3\cdot 0.85^{3}\cdot 10^{-\frac{4}{5}}+0.85^{3}\cdot 10^{-\frac{7}{5}})\cdot 10^{\frac{n-1}{5}} +f_1(n-1)\\
&< 0.9475\cdot 10^{\frac{n-1}{5}}+f_1(n-1)<f(n),
\end{aligned}
\end{equation*}
a contradiction.
Finally, we exclude $s=2$  (i.e., there are two components in $H=G-K_5$).   Denote by $|N_{H_i}(v)|=d_i$. Assume $\min \{d_1,d_2\}=1$, and w.l.o.g., let $d_1=1$.
Recall that $v$ is the cutpoint of pendant block $K_5$.
Consider
\begin{equation*}
\begin{aligned}
|M_G|&\leq |M_{G-v}|+\sum_{p\in N(v)}|M_{G-N[v]\cup N[p]}|\\
&\leq |M_{K_4}|\cdot 0.85^{2}\cdot g(n_1)\cdot g(n_2) +4\cdot 0.85\cdot g(n_1-1)\cdot g(n_2-d_2)\\
&+ g(n_1-2)\cdot g(n_2-d_2)+  0.85\cdot g(n_1-1)\cdot d_2\cdot g(n_2-d_2)\\
&\leq (6\cdot 0.85^{2}\cdot10^{-\frac{4}{5}}+4\cdot 0.85 \cdot 10^{-\frac{6}{5}}+10^{-\frac{7}{5}}\cdot 0.85 \cdot 10^{-\frac{6}{5}})\cdot 10^{\frac{n-1}{5}}\\
&< 0.9951\cdot 10^{\frac{n-1}{5}}<f(n),
\end{aligned}
\end{equation*}
a contradiction.
Assume $\min \{d_1,d_2\}\geq 2.$ It is easy to see that
\begin{equation*}
\begin{aligned}
|M_G|&\leq |M_{G-v}|+\sum_{p\in N(v)}|M_{G-N[v]\cup N[p]}|\\
&\leq  |M_{K_4}|\cdot 0.85^{2}\cdot g(n_1)\cdot g(n_2) + (4+d_1+d_2)\cdot g(n_1-d_1)\cdot g(n_2-d_2)<f(n),
\end{aligned}
\end{equation*}
which contradicts the fact that $G$ is a counterexample and completes the
proof of Lemma \ref{bopen05}.
\end{proof}

%2\cdot10^{\frac{n-7}{5}}+10^{\frac{n-8}{5}}+2\cdot10^{\frac{n-6}{5}}+f_1(n-1)\\
%       & <10^{\frac{n-1}{5}}+f_1(n)=f(n),
\begin{lemma} \label{bopen004}
Let $G$ be a minimal counterexample.  Then $G$ does not have a pendant   block $K_4$.
\end{lemma}
\begin{proof}
Similar to the proof of Lemma  \ref{bopen06}.
For contradiction, let $K_4$ be a pendant   block  with cutpoint $v$.
By Lemma \ref{ab}, it suffices to exclude the case that $d(v)\geq 5$.
By Lemma \ref{bo1} $(1)$, $G-K_4$ is disconnected. Let $H_1,H_2,\ldots, H_s$ be all components of  $H=G-K_4$ with $|V(H_i)|=n_i$ and $s\geq 2$.

%Similar to the above two lemmas,

\noindent{\bf Claim 1.} {\it There are neither subgraphs of order $4, 5$ nor $K_6, K_6-e$ as a component in $H$.}

Note that $v$ is a cutpoint of $G$. Then $v+H_i= K_{n_i+1}$ if $n_i\in \{4,5\}$ by Lemmas \ref{gcut5}  and
\ref{gcut6}, which contradicts   Lemma \ref{bopen06} and Lemma \ref{bopen05}, respectively.
Applying the proof of Claim $1$ in Lemma \ref{bopen06}, we can also exclude $K_6$ and $K_6-e$.

\noindent{\bf Claim 2.}   {\it $|M_{H_i-N_{H_i}(v)}|\leq 0.85g(n_i-1)$ for any $1\leq i\leq s.$}

If $|N_{H_i}(v)|\geq 2$, then $|M_{H_i-N_{H_i}(v)}|\leq g(n_i-2)<0.85g(n_i-1).$
If $|N_{H_i}(v)|=1$, denoted by $N_{H_i}(v)=w_i$, then  $w_i$ is a cutpoint of $G$.
Similar to the proof of  Claim 1, we obtain that there are neither subgraphs of order $4, 5$ nor $K_6, K_6-e$ as a component in the subgraph $H_i-N_{H_i}(v)$. Thus  $|M_{H_i-N_{H_i}(v)}|\leq 0.85g(n_i-1)$ follows by
  Lemma \ref{botool1}. Claim $2$ follows.

Let $u\in N_{K_5}(v).$ Now consider the upper bound on  $|M_G|\leq |M_{G-u}|+\sum_{p\in N(u)}|M_{G-N[u]\cup N[p]}|.$ Clearly, $|M_{G-u}|\leq f(n-1)$. If $p\neq v$, then $|M_{G-N[u]\cup N[p]}|\leq 0.85^{s}\cdot g(n-4).$
If $p= v$, then $|M_{G-N[u]\cup N[p]}|\leq 0.85^{s}\cdot g(n-4-s).$
Putting all this together,
\begin{equation*}
\begin{aligned}
|M_G|\leq f(n-1)+2\cdot 0.85^{s}\cdot g(n-4)+0.85^{s}\cdot g(n-4-s).
\end{aligned}
\end{equation*}
When $s\geq 3$,
\begin{equation*}
\begin{aligned}
|M_G|&\leq (10^{-\frac{1}{5}}+2\cdot 0.85^{3}\cdot 10^{-\frac{3}{5}}+0.85^{3}\cdot 10^{-\frac{6}{5}})\cdot 10^{\frac{n-1}{5}} +f_1(n-1)\\
&< 0.9783\cdot 10^{\frac{n-1}{5}}+f_1(n-1)<f(n),
\end{aligned}
\end{equation*}
a contradiction.
Finally, similar to the proof of Lemma  \ref{bopen05}, it is easy to exclude $s=2.$   The result follows.
\end{proof}

Combining Lemma \ref{gpb3} with the above three lemmas, we obtain the structural information of $G$ as follows.
\begin{corollary} \label{forbidden}
Let $G$ be a minimal counterexample.  Then $G$ does not have the  pendant   blocks  $K_3, K_4, K_5, K_6$.
\end{corollary}

Up to now, the proof of Property C is complete.

\subsection{\textbf{Property D: each minimal counterexample has no non-cutpoints}\label{subsmall}}
In this subsection, we analyze the substructure induced by non-cutpoints and their neighbors, which will be crucial for our final main theorem.
\begin{lemma}  \label{z1}
Let $G$ be a minimal counterexample. If $u$ is a non-cutpoint  with $d(u)=5$, then
 there is exactly one neighbor $v\in N(u)$ such that $N[v]\backslash N[u] \neq \emptyset$.
\end{lemma}
\begin{proof}
Let $X:=\{v\in N_G(u)~|~|N[v]\setminus N[u]|\ge 1\}$.
Clearly, the result follows if $|X|=1$. Then we will prove that $|M_G|\leq f(n)$ if $|X|\geq 2$, which contradicts the fact that $G$ is a counterexample.
By Proposition \ref{bopro1},
\begin{equation}\label{noncut5}
\begin{aligned}
|M_G|&\leq |M_{G-u}|+\sum_{p\in N(u)}|M_{G-N[u]\cup N[p]}|.
\end{aligned}
\end{equation}
Since $u$ is a non-cutpoint, $G-u$ is connected and further $|M_{G-u}|\leq f(n-1)$. If  $p\in N_G(u)$ and $N[p]\subseteq N[u]$  (i.e., $p\notin X$), then $|M_{G-N[u]\cup N[p]}|\leq g(n-6)$;  if $|N[p]\backslash N[u]|\ge a\ge 1$ (i.e.,  $p\in X$), then $|M_{G-N[u]\cup N[p]}|\leq g(n-6-a)$. We divide the proof into three cases.

\noindent {\bf Case 1.} {\it $|X|\ge 4.$}

In this case, Ineq. (\ref{noncut5}) yields
\begin{equation*}
\begin{aligned}
|M_G|&\leq  f(n-1)+4g(n-7)+g(n-6)\\
&= (10^{-\frac{1}{5}}+4\cdot  10^{-\frac{6}{5}}+ 10^{-\frac{5}{5}})\cdot 10^{\frac{n-1}{5}}+f_1(n-1)\\
&<0.9834\cdot 10^\frac{n-1}{5}+f_1(n-1)<f(n),
\end{aligned}
\end{equation*}
contradicting
the fact that $G$ is a counterexample.

\noindent {\bf Case 2.} {\it $|X|=3.$}

If there exists a vertex $v\in X$ such that $|N[v]\backslash N[u]|\ge 2$, then
Ineq. (\ref{noncut5}) yields
\begin{equation*}
\begin{aligned}
|M_G|&\leq  f(n-1)+2g(n-7)+g(n-8)+2g(n-6)\\
&= (10^{-\frac{1}{5}}+2\cdot  10^{-\frac{6}{5}}+ 10^{-\frac{7}{5}}+2\cdot  10^{-\frac{5}{5}})\cdot 10^{\frac{n-1}{5}}+f_1(n-1)\\
&<0.9970\cdot 10^\frac{n-1}{5}+f_1(n-1)<f(n),
\end{aligned}
\end{equation*}
 a contradiction.

Now we exclude the case where   $|N[v]\backslash N[u]|=1$ for each $v\in X$.
In this situation, the subgraph $G-N[u]$ has at most three  components, denoted by
$G_1, \cdots, G_j$, where $1\le j\le3$.
Let $n_k$ be the order of $G_k$.
If  $N[p]\subseteq N[u]$, then $|M_{G-N[u]\cup N[p]}|=|M_{G-N[u]}|=\prod_{k=1}^{j}|M_{G_k}|  \leq\prod_{k=1}^{j}f(n_k)$. Therefore, Ineq. (\ref{noncut5})  yields
\begin{equation}\label{noncut5case22}
\begin{aligned}
|M_G|&\leq f(n-1)+3g(n-7)+2\prod_{k=1}^{j}f(n_k),
\end{aligned}
\end{equation}
where $\sum_{k=1}^jn_k=n-6$.
We claim that there is a component $G_k~(1\le k\le j)$ such that $n_k\notin\{4, 5, 6\}$.
If $j=1,$ i.e., $G-N[u]$ has exactly one component, then $|G|\geq 15$ and $|N[u]|=6$ forces that $n_k\notin\{4, 5, 6\}.$
When $j\geq 2,$ since $|N[v]\backslash N[u]|=1$ for each $v\in X$, there is a component $G_k$ with
$|N(G_k)\cap X|=1$, and denote this common vertex by $v_k=N(G_k)\cap X$, i.e. $v_k\in X$.
Recall that $|N[v_k]\backslash N[u]|=1$ and let  $w_k=N(v_k)\cap V(G_k)$.
Since $w_k$ is a cutpoint of $G$,  $n_k\notin\{4, 5, 6\}.$ Otherwise,  $G_k$ is the pendant   block
$K_{n_k}$ by Lemmas \ref{gcut5} and \ref{gcut6}, which   contradicts Corollary  \ref{forbidden}.
Further, $f(n_k)\leq 0.85g(n_k)$ due to Lemma \ref{botool1}. Combining this upper bound, Ineq. (\ref{noncut5case22})  yields
\begin{equation*}
\begin{aligned}
|M_G|&\leq  f(n-1)+3g(n-7)+2\cdot 0.85\cdot g(n-6)\\&= (10^{-\frac{1}{5}}+3\cdot  10^{-\frac{6}{5}}+ 2\cdot 0.85\cdot 10^{-\frac{5}{5}})\cdot 10^{\frac{n-1}{5}}+f_1(n-1)\\
&<0.9903\cdot 10^\frac{n-1}{5}+f_1(n-1)<f(n),
\end{aligned}
\end{equation*}
 a contradiction.
%\begin{equation*}%\label{zbo1}
%\begin{aligned}
%|M_G|&\le  f(n-1)+3g(n-7)+2\cdot0.85\cdot g(n-6)  \\
%       & = 10^{\frac{n-2}{5}}+3\cdot10^{\frac{n-7}{5}}
%      +2\cdot0.85\cdot10^{\frac{n-6}{5}}\\
%       & <10^{\frac{n-1}{5}}+f_1(n)=f(n),
%\end{aligned}%\eqno{(3.2)}
%\end{equation*}
%which contradict the choice of $G$.

\noindent {\bf Case 3.} {\it $|X|=2.$}

\noindent {\bf Claim 1.} {\it If $G_k$ is a component of $G-N[u]$ with $|G_k|\in\{4, 5, 6\}$, then $|N(X)\cap V(G_k)|\geq 2.$}

Otherwise, let $N(X)\cap V(G_k)=v_k$ where $v_k\in V(G_k).$ Then $v_k$ is a cutpoint and  thus $G_k$ is one of the pendant blocks $K_4, K_5, K_6$ by Lemmas \ref{gcut5} and \ref{gcut6}, which contradicts Corollary  \ref{forbidden}.

\noindent {\bf Claim 2.} {\it  There are at least two components in $G-N[u]$.}

For contradiction, suppose that there is exactly one component in $G-N[u]$. Now  Ineq. (\ref{noncut5})  yields
\begin{equation}\label{noncut5case31}
\begin{aligned}
|M_G|&\leq  f(n-1)+2 g(n-7) +3f(n-6).
\end{aligned}
\end{equation}
For $15\leq |G|\leq 18,$ it can be directly  calculated  that $|M_G|\leq f(n)$.
For $|G|\geq 19,$ $\frac{f(n-6)}{g(n-6)}\leq \frac{f(14)}{g(14)}< 0.7778$ by   Lemma \ref{botool1}.
Replacing $f(n-6)$ with $0.7778g(n-6)$ in Ineq. (\ref{noncut5case31}), it is easy to see that
\begin{equation*}
\begin{aligned}
|M_G|&\leq  f(n-1)+2 g(n-7) +3f(n-6)\\
&<(10^{-\frac{1}{5}}+ 2\cdot 10^{-\frac{6}{5}}+3\cdot 0.7778\cdot10^{-\frac{5}{5}})\cdot 10^{\frac{n-1}{5}} +f_1(n-1)\\
&<0.9905\cdot 10^{\frac{n-1}{5}} +f_1(n-1)< f(n).
\end{aligned}
\end{equation*}
 Both of the two cases contradict the choice of $G$. So Claim 2 follows.

\noindent {\bf Claim 3.} {\it  Every  component  of $G-N[u]$ is of order  equal to one of $\{4, 5, 6\}$.}

For contradiction, first we assume that there are at least two components of $G-N[u]$ of order not equal to $\{4, 5, 6\}$. Then replacing $f(n-6)$ with $0.85^{2}g(n-6)$ in Ineq. (\ref{noncut5case31}) due to Lemma \ref{botool1}, we obtain that
\begin{equation*}
  \begin{aligned}
  |M_G|&\leq  f(n-1)+2 g(n-7) +3f(n-6)\\
  &\leq  (10^{-\frac{1}{5}}+ 2\cdot 10^{-\frac{6}{5}}+3\cdot 0.85^{2}\cdot 10^{-\frac{5}{5}})\cdot 10^{\frac{n-1}{5}} +f_1(n-1)\\
  &<0.9739\cdot 10^{\frac{n-1}{5}} +f_1(n-1)< f(n),
  \end{aligned}
\end{equation*}
 a contradiction. Finally, we assume that there is exactly one component, say $G_1$, of order not equal to $\{4, 5, 6\}$. Then  Claim 2 implies that there is another component, say $G_2$,  of order equal to one of $\{4, 5, 6\}$.   Claim 1 forces that
$|N(X)\setminus N[u]|\geq 3$. Consequently, we obtain that
\begin{equation*}
  \begin{aligned}
|M_G|&\leq f(n-1)+ g(n-8) + g(n-7) + 3\cdot 0.85\cdot g(n-6)\\
&=(10^{-\frac{1}{5}}+ 10^{-\frac{7}{5}}+ 10^{-\frac{6}{5}}+3\cdot 0.85 \cdot 10^{-\frac{5}{5}})\cdot 10^{\frac{n-1}{5}}+f_1(n-1)\\
&<0.9889\cdot 10^{\frac{n-1}{5}} +f_1(n-1)< f(n),
  \end{aligned}
\end{equation*}
a contradiction. So Claim 3 follows.

Label the vertices of $X$ by $v_1, v_2$ in which
$|N[v_1]\setminus N[u]|\geq |N[v_2]\setminus N[u]|$.
We claim that there are exactly two cases: either $|N[v_1]\setminus N[u]|=|N[v_2]\setminus N[u]|=2$; or
 $|N[v_2]\setminus N[u]|=1$ and $|N[v_1]\setminus N[u]|\leq 6$.
Otherwise, if $|N[v_1]\setminus N[u]|\ge 3$ and $|N[v_2]\setminus N[u]|\ge 2$, then
$|M_G|\leq  f(n-1)+g(n-9)+g(n-8)+3g(n-6)< f(n)$
%\begin{equation*}
%  \begin{aligned}
%  |M_G|\leq  f(n-1)+g(n-9)+g(n-8)+3g(n-6)< f(n),
%  \end{aligned}
%\end{equation*}
 by a direct calculation;
if $|N[v_1]\setminus N[u]|\ge 7$ and $|N[v_2]\setminus N[u]|\ge 1$, then
$|M_G|\leq  f(n-1)+g(n-13)+g(n-7)+3g(n-6)< f(n)$
%\begin{equation*}
%  \begin{aligned}
%  |M_G|\leq  f(n-1)+g(n-13)+g(n-7)+3g(n-6)< f(n),
%  \end{aligned}
%\end{equation*}
by a direct calculation. The claim follows.

\noindent {\bf Claim 4.} {\it For $|X|=2$, if $|N[v_1]\setminus N[u]|=|N[v_2]\setminus N[u]|=2$, then $|M_G|\leq f(n)$.}

Combining Claims $1, 2$ with  $3$, we conclude that there are exactly two components $G_k$ with $n_k\in\{4, 5, 6\}$. Note that $|N[u]|=6$ and then   $15\leq |G|\leq 18$. It can be  directly  calculated  that $|M_G|\leq f(n)$.

\noindent {\bf Claim 5.} {\it For $|X|=2$, if $|N[v_1]\setminus N[u]|\leq 6$ and  $|N[v_2]\setminus N[u]|=1$, then $|M_G|\leq f(n)$.}

W.l.o.g.,   denote by $|N(v_2)\cap V(G_1)|=1$, then $|N(v_1)\cap V(G_1)|\geq 1$ due to Claims 1 and 3.
Now only $v_1$ is adjacent to the other components of $G-N[u]-G_1$  and it is easy to see that $v_1$
is a cutpoint of $G$. Similar to the above discussion, for any other component $G_k~(i\neq 1)$ of $G-N[u]$, it holds that $n_k\notin \{4, 5\}$.
%Otherwise, $v_1+G_k$ is one the pendant blocks $K_5, K_6,$, a contradiction.
The same as the proof of Claim 1 in Lemma  \ref{bopen06}, we can also obtain that $G_k\notin \{K_6, K_6-e\}$.
Combining Claim $3$, we have that $|M_{G_k}|\leq \frac{13}{10^{\frac{6}{5}}}\cdot g(6)$.
Consequently, Ineq. (\ref{noncut5})  yields
$|M_G|\leq  f(n-1)+ g(n-7) + g(n-9)+3\cdot \frac{13}{10^{\frac{6}{5}}}\cdot g(n-6)< f(n),$
%\begin{equation*}
%  \begin{aligned}
%  |M_G|\leq  f(n-1)+ g(n-7) + g(n-9)+3\cdot \frac{13}{10^{\frac{6}{5}}}\cdot g(n-6)< f(n),
%  \end{aligned}
%\end{equation*}
a contradiction.

Hence $|X|=1$, as desired.
\end{proof}

\begin{corollary}  \label{d5noncut}
Let $G$ be a minimal counterexample and $u$ be a non-cutpoint of  $G$.  Then $2\leq d(u)\leq 4$.
\end{corollary}
\begin{proof}
By Lemmas \ref{bolem1} and \ref{bolem6c}, it suffices to exclude the case that $d(u)=5$. If not, Lemma \ref{z1} shows
that there is exactly one neighbor $v\in N(u)$ such that $N[v]\backslash N[u] \neq \emptyset$. Note that $v$ is a cutpoint of $G$, and thus the subgraph $G[N[u]]=K_6$ by Lemma \ref{gcut6}, which  contradicts Corollary \ref{forbidden}.
\end{proof}

\begin{lemma} \label{d43tool}
Let $G$ be a minimal counterexample and $u$ be any vertex of  $G$.  Then for any neighbor $p\in N(u)$, there is no  $K_4,  K_5, K_5-e, K_6, K_6-e$ as a component in the subgraph $G-N[u]\cup N[p]$.
\end{lemma}
\begin{proof}
For contradiction, suppose there is a component $H\in \{K_4,  K_5, K_5-e, K_6, K_6-e\}$ in $G-N[u]\cup N[p]$ for some $p\in N(u)$. We divide the proof into the following two cases.

\noindent {\bf Case 1.} {\it  $|N(G-H)\cap V(H)|=1$.}

W.l.o.g., let  $N(G-H)\cap V(H)= \{v\}$ where $v\in V(H)$,  then $v$ is a cutpoint of $G$. Further, $H$ is one of the pendant blocks $K_4, K_5, K_6$ by Lemmas   \ref{gcut5} and \ref{gcut6}, which
contradicts Corollary \ref{forbidden}.

\noindent {\bf Case 2.} {\it  $|N(G-H)\cap V(H)|\geq 2$.}

Note that in this case every vertex $v\in V(H)$   is a non-cutpoint of $G$ and thus $2\leq d_{G}(v)\leq 4$ by Corollary   \ref{d5noncut}.
When $H\in \{K_6, K_6-e, K_5\}$, there is a non-cutpoint $v\in V(H)$  of $G$ with $d_{G}(v)\geq 5$,  a contradiction.
When $H=K_5-e~(e=vw)$, in order to exclude non-cutpoints of degree greater than   $4$, then $N(G-H)\cap V(H)=\{v,w\}$ with $d_G(v)=d_G(w)= 4$. Besides, $|N(H)\cap V(G-H)|=2;$ otherwise, if $|N(H)\cap V(G-H)|=1,$
then the subgraph $G[N(H)+V(H)]=K_6$ by Lemma \ref{gcut6}, a contradiction.
Now we construct a new graph $G'$ obtained from $G$ by just adding an edge $vw$ and it is easy to verify  that $|M_{G'}|\geq|M_G|>f(n)$ and $\tau(G')= \tau(G)$. This means that $G'$ is also a minimal counterexample; however, there is a non-cutpoint  $v\in V(G')$ of degree $5$, which
contradicts Corollary \ref{d5noncut}.

It remains to  exclude the case where $H=K_4$ and  $|N(G-H)\cap V(H)|\geq 2$. Observe that $|N(H)\cap V(G-H)|\geq2;$ otherwise there is a pendant $K_5$, a contradiction. For convenience, we label  $V(H)=\{a,b,c,d\}$ and divide into three cases to consider the possibilities.
We first consider $|N(G-H)\cap V(H)|=2$, w.l.o.g. denoted by $N(G-H)\cap V(H)=\{a,b\}$, in order to exclude non-cutpoints of degree greater than   $4$, then $d_G(a)=d_G(b)= 4$. Further, combining $|N(H)\cap V(G-H)|\geq2$,  we also have that $|N(H)\cap V(G-H)|= 2$. Then applying Lemma \ref{y1} to the edge $ab$, we obtain that either
$|M_{G_{a\rightarrow b}}|\geq|M_G|>f(n)$ with $\tau(G_{a\rightarrow b})< \tau(G)$ or $|M_{G_{b\rightarrow a}}|\geq|M_G|>f(n)$ with $\tau(G_{b\rightarrow a})< \tau(G)$. However, both of the two cases contradict  the choice of $G$.

We now consider $|N(G-H)\cap V(H)|=3$, w.l.o.g. denoted by
$N(G-H)\cap V(H)=\{a,b,c\}$, we can also obtain that $d_G(a)=d_G(b)= d_G(c)=4$ similarly. For
the set $\{N(v)\setminus V(H)~|~v\in \{a,b,c\}\}$, no matter they are all different or exactly two of them are different (w.l.o.g. $N(a)\setminus V(H)\neq N(b)\setminus V(H)$), Lemmas \ref{y-y1} and \ref{twins} give that either $|M_{G_{T_G(a)\rightarrow b}}|\geq|M_G|>f(n)$ with $\tau(G_{T_G(a)\rightarrow b})< \tau(G)$ or $|M_{G_{T_G(b)\rightarrow a}}|\geq|M_G|>f(n)$ with $\tau(G_{T_G(b)\rightarrow a})< \tau(G)$, contradicting the choice of $G$.

Finally, we consider $N(G-H)\cap V(H)=\{a,b,c,d\}$, then  $d_G(a)=d_G(b)= d_G(c)=d_G(d)=4$ similarly.
Now we divide into two cases to analyze  the set $\{N(v)\setminus V(H)~|~v\in \{a,b,c,d\}\}$. We first claim that it can not happen the case that exactly two of them are different, (w.l.o.g. $N(a)\setminus V(H)\neq N(b)\setminus V(H)$). If not,
Lemmas \ref{y-y1} and \ref{twins} give that either $|M_{G_{T_G(a)\rightarrow b}}|\geq|M_G|>f(n)$ with $\tau(G_{T_G(a)\rightarrow b})< \tau(G)$ or $|M_{G_{T_G(b)\rightarrow a}}|\geq|M_G|>f(n)$ with $\tau(G_{T_G(b)\rightarrow a})< \tau(G)$, contradicting the choice of $G$.
For the remaining cases, it can readily be verified that there exists a vertex, say $a\in V(H)$, such that for any neighbor $p\in N(a)$, $N(p)\nsubseteq N(a)$ and thus $|G-N[a]\cup N[p]|\leq n-6$.
Consider
\begin{equation}\label{0.85}
\begin{aligned}
|M_G|\leq  |M_{G-a}|+\sum_{p\in N(a)}|M_{G-N[a]\cup N[p]}|.
\end{aligned}
\end{equation}
%$|M_G|\leq |M_{G-a}|+\sum_{p\in N(a)}|M_{G-N[a]\cup N[p]}|.$
Applying  the same discussion above to the vertex $a$, it is easy to verify that there is no $K_5$ as a component in $G-N[a]\cup N[p]$ for any $p\in N(a)$.
Further for any component $S$ of order $n_s$ in $G-N[a]\cup N[p]$, $|M_S|\leq f(n_s)\leq 0.951g(n_s)$  by Lemma \ref{botool1}.
If there exists $t~(1\le t\le 4)$ vertices of $N(a)$ such that $G-N[a]\cup N[p]$ is connected, then when
$n\geq20$ Ineq.
\ref{0.85} yields $|M_G|\leq f(n-1)+tf(n-6)+(4-t)\cdot 0.951^2g(n-6)\leq f(n).$
When $15\leq n<20$, by calculation, $|M_G|\leq f(n).$
This implies that  for any $p\in N(a)$, $G-N[a]\cup N[p]$ is disconnected, and further
$|M_G|\leq f(n-1)+4\cdot0.951^2\cdot 10^\frac{n-6}{5}<f(n).$
This contradicts the fact that $G$ is a counterexample and completes the
proof of Lemma \ref{d43tool}.
\end{proof}

\begin{lemma}\label{yyn2}
Let $G$ be a minimal counterexample. Then $G$ does not have a non-cutpoint of degree $2$.
\end{lemma}
\begin{proof} Due to  Lemma \ref{y1}, the minimum degree of $G$  is at least  $2$.
For contradiction that $G$ has a non-cutpoint $u$ of degree $2$, then for any $v\in N(u)$ with
$d(v)=2$, $v$ is also a non-cutpoint of $G$. Since  $G\neq C_n$, there exists a non-cutpoint $x$ of degree $2$
which has a neighbor  $y\in N(x)$ satisfying $d(y)\geq 3.$
Consider
\begin{equation} \label{noncut2}
\begin{aligned}
|M_{G}|\leq |M_{G-x}|+\sum_{p\in N(x)}|M_{G-N[x]\cup N[p]}|.
\end{aligned}
\end{equation}
Let $N(x)= \{y,z\}.$
If $d(z)=2,$ then $y$ and  $z$   are not adjacent in $G$ by Lemma \ref{gpb3}.
So  $|G-N[x]\cup N[z]|\leq n-4$ and further $|M_{G-N[x]\cup N[z]}|\leq f(n-4)\leq 0.85g(n-4)$ when
$G-N[x]\cup N[z]$ is connected and  $n-4\geq 10$ by Lemma \ref{botool1}; $|M_{G-N[x]\cup N[z]}|\leq  0.85^{2} g(n-4)<0.85 g(n-4)$ when $G-N[x]\cup N[z]$ is disconnected by Lemmas \ref{botool1} and \ref{d43tool}. Similarly, $|M_{G-N[x]\cup N[y]}|\leq 0.85g(n-5)$   whenever  $G-N[x]\cup N[y]$ is connected or not.
Recall that $x$ is  a non-cutpoint, then  when $n\geq 15$,
Ineq. (\ref{noncut2}) yields
\begin{equation*}
\begin{aligned}
|M_G|&\leq  f(n-1)+ 0.85g(n-4)+0.85g(n-5)\\
&<0.9792\cdot 10^\frac{n-1}{5}+f_1(n-1) <f(n),
\end{aligned}
\end{equation*}
  a contradiction.

If $d(z)\geq 3,$ then $|G-N[x]\cup N[z]|\leq n-4$ and further $|M_{G-N[x]\cup N[z]}|\leq f(n-4)\leq 0.7218g(n-4)< 0.85^{2} g(n-4)$ when
$G-N[x]\cup N[z]$ is connected and  $n-4\geq 19$ by   Lemma \ref{botool1}; $|M_{G-N[x]\cup N[z]}|\leq  0.85^{2} g(n-4)$ when $G-N[x]\cup N[z]$ is disconnected by Lemmas \ref{botool1} and \ref{d43tool}. Similarly, since $d(y)\geq 3,$ $|M_{G-N[x]\cup N[y]}|\leq 0.85^{2} g(n-4)$ when $n-4\geq 19.$
Consequently, when $n\geq 23$
Ineq. (\ref{noncut2}) yields
\begin{equation*}
\begin{aligned}
|M_G|&\leq  f(n-1)+ 2\cdot 0.85^{2}\cdot g(n-4)\\
&<0.9940\cdot 10^\frac{n-1}{5}+f_1(n-1) <f(n),
\end{aligned}
\end{equation*}
a contradiction.
For
$15\leq n\leq 22,$  $|M_G|\leq f(n)$ by a direct verification.
This contradicts the fact that $G$ is a counterexample.
\end{proof}

\begin{lemma}   \label{zy2}
Let $G$ be a minimal counterexample.  If $u$ is a
non-cutpoint  with $d(u)=4$, then there is exactly one neighbor $v\in N(u)$ such that $N[v]\backslash N[u] \neq \emptyset$.
\end{lemma}
\begin{proof} Recall that
\begin{equation}\label{noncut4case1}
\begin{aligned}
|M_G|&\leq  |M_{G-u}|+\sum_{p\in N(u)}|M_{G-N[u]\cup N[p]}|.
\end{aligned}
\end{equation}
%In order to obtain a better upper bound on $|M_G|$, we first give the following claim.
%{\bf Claim 2.} {\it  $|M_G|\leq f(n)$.}\\
By Lemmas \ref{botool1} and \ref{d43tool}, it holds that for any $p\in N(u)$ and
any component $C$ on  $n_c$  vertices of the subgraph $G-N[u]\cup N[p]$, $|M_C|\leq f(n_c)\leq 0.85g(n_c).$
Denote by $N(u)=\{v_1,v_2,v_3,v_4\}$ and $X:=\{v_i\in N(u)~|~|N(v_i)\backslash N[u]|\geq1\}$.
We claim that if $|X|\geq 2$, then $|M_G|\leq f(n)$, contradicting the choice of $G$.

\noindent{\bf Case 1.} {\it  $|X|=3$ and let $X=\{v_1,v_2,v_3\}.$}

Note that $N[v_4]\subseteq N[u]$, if $G-N[u]$  is connected, then
$|M_{G-N[u]\cup N[v_4]}|\leq f(n-5)<0.7218g(n-5)<0.85^{2}g(n-5)$ when $n-5\geq 19$ by Lemma \ref{botool1};  if $G-N[u]$  is disconnected, then $|M_{G-N[u]\cup N[v_4]}|\leq 0.85^{2}g(n-5)$. Thus when $n-5\geq 19$, $|M_{G-N[u]\cup N[v_4]}|\leq 0.85^{2}g(n-5)$ no matter whether  $G-N[u]$  is connected or not.
Similarly, when $n-6\geq 19$,
for each $v_i\in X$,  $|M_{G-N[u]\cup N[v_i]}|\leq 0.85^{2}g(n-6)$.
Together with the fact that  $u$ is a non-cutpoint,   when $n\geq 25$, Ineq. (\ref{noncut4case1}) yields $|M_G|\leq  f(n-1)+3\cdot 0.85^{2}g(n-6)+0.85^{2}g(n-5)<0.9623\cdot 10^\frac{n-1}{5}+f_1(n-1)<f(n).$
When $15\leq n\leq 24$, $|M_G|\leq f(n)$ by a direct verification.
It follows that $|X|\neq 3$.

\noindent {\bf Case 2.} {\it  $|X|=4$.}

It is easy to see that the upper bound on $|M_G|$ in this case is smaller than that in Case 1. This implies that  $|M_G|\leq f(n)$ if $|X|=4$, which contradicts the choice of $G$.

\noindent {\bf Case 3.} {\it  $|X|=2$ and let $X=\{v_1,v_2\}$.}

Suppose first that $G-N[u]$ is  connected, since $N[v_3]\subseteq N[u]$,
$|M_{G-N[u]\cup N[v_3]}|\leq f(n-5)$.
Similar to the discussion in Case 1, when $n\geq 25$, Ineq. (\ref{noncut4case1}) yields
$|M_G|\leq f(n-1)+f(n-5)+0.85^{2}g(n-5)+2\cdot 0.85^{2}g(n-6)<0.9900\cdot 10^\frac{n-1}{5}+f_1(n-1)+ f_1(n-5)<f(n)$. When $15\leq n\leq 24$, $|M_G|\leq f(n)$ by a direct verification.

Suppose now that $G-N[u]$ is  disconnected.  If there is some $v_i\in X$  such that $|N(v_i)\backslash N[u]|\geq 2,$ then when $n\geq 26$, Ineq. (\ref{noncut4case1}) yields
$|M_G|\leq f(n-1)+2\cdot0.85^{2}g(n-5)+ 0.85^{2}g(n-7)+0.85^{2}g(n-6)<f(n)$.
When $15\leq n\leq 25$, $|M_G|\leq f(n)$ by a direct verification.
Now consider  $|N(v_1)\backslash N[u]|=|N(v_2)\backslash N[u]|=1.$
In this case, $v_1, v_2$ are both cutpoints  and $N(v_1)\backslash N[u]\neq N(v_2)\backslash N[u]$.
The facts that $u$ is non-cutpoint, $N[v_3]\subseteq N[u]$ and  $N[v_4]\subseteq N[u]$ imply that  $v_3, v_4$ are both non-cutpoints. So $d_G(v_3),d_G(v_4)\in \{3,4\}$ by Corollary \ref{d5noncut} and
Lemma \ref{yyn2}. First, assume  that   one of $v_3$ and $v_4$, say $v_3$, satisfies $d_G(v_3)=3$.
Then $v_3v_4\in E(G)$, otherwise $v_3, v_4$ are false twins which contradicts Lemma  \ref{false}.
If  $v_3$ and $v_4$ are twins, then  $v_1v_2\in E(G)$ as $G-u$ is connected.
Applying Lemma  \ref{y-y1} to the edge $uv_3$, we obtain that either $|M_{G_{u\rightarrow v_3}}|\geq|M_G|$
with $\tau(G_{u\rightarrow v_3})< \tau(G)$ or $|M_{G_{T_{G}(v_3)\rightarrow u}}|\geq|M_G|$ with $\tau(G_{T_{G}(v_3)\rightarrow u})< \tau(G)$. Both of the two cases contradict the choice of $G$.
Assume now that $v_3$ and $v_4$ are not twins, then one can see that for any neighbor  $p\in N(v_3)$,
$|G-N[v_3]\cup N[p]|\leq n-5.$
Further, it is easy to verify that
$|M_G|\leq |M_{G-v_3}|+\sum_{p\in N(v_3)}|M_{G-N[v_3]\cup N[p]}|\leq f(n-1)+3\cdot0.85^{2}g(n-5)<f(n).$
%\begin{equation*}
%\begin{aligned}
%|M_G|&\leq |M_{G-v_3}|+\sum_{p\in N(v_3)}|M_{G-N[v_3]\cup N[p]}|\\
%&\leq f(n-1)+3\cdot0.85^{2}g(n-5)<f(n).
%\end{aligned}
%\end{equation*}

Finally, it suffices to consider the remaining case that $d_G(v_3)=d_G(v_4)=4.$
If $v_1v_2\in E(G)$, by Lemmas \ref{botool1} and \ref{d43tool}, we have
\begin{align*}
  |M_G| & \leq |M_G(\{v_1\},\{v_2\})|+| M_G(\{v_2\},\{v_1\})| +|M_G(\emptyset,\{v_1, v_2\})| +|M_G(\{v_1, v_2\}, \emptyset)|\\
   & \leq 2\cdot  \left( 3\cdot 0.85^2g(n-6)+ g(n-7)\right)+ 0.85^2g(n-7)+3\cdot 0.85^2g(n-5)\\
   &<0.9489\cdot 10^\frac{n-1}{5} < f(n),
\end{align*}
contradicting the choice of $G$.
If $v_1v_2\notin E(G)$, by Lemmas \ref{botool1} and \ref{d43tool}, we have
\begin{align*}
  |M_G| & \leq |M_G(\{v_1\},\{v_2\})|+| M_G(\{v_2\},\{v_1\})| +|M_G(\emptyset,\{v_1, v_2\})| +|M_G(\{v_1, v_2\}, \emptyset)|\\
   & \leq 2\cdot  \left( 3\cdot 0.85^2g(n-6)+ g(n-7)\right)+ g(n-9)+3\cdot 0.85^2g(n-5)\\
   &<0.9284\cdot 10^\frac{n-1}{5} < f(n),
\end{align*}
yielding the desired contradiction. Thus $|X|\neq 2$.
This completes the proof of Lemma \ref{zy2}.
\end{proof}

\begin{corollary}  \label{d54noncut}
Let $G$ be a minimal counterexample.  Then  each non-cutpoint of $G$ is degree $3$.
\end{corollary}
\begin{proof}
Combining Corollary \ref{d5noncut} with Lemma \ref{yyn2}, we deduce that the degree of  non-cutpoints is either $3$ or $4$. Lemma \ref{zy2} further specifies that if $u$ is a non-cutpoint of degree $4$, then there is exactly one neighbor $v\in N(u)$ such that $N[v]\backslash N[u] \neq \emptyset$. Observe that $v$ is a cutpoint of $G$, and thus the subgraph $G[N[u]]= K_5$ by Lemma \ref{gcut5}, contradicting Lemma \ref{bopen05}. The result follows.
\end{proof}

\begin{lemma}   \label{zy3}
Let $G$ be a minimal counterexample.  If $u$ is a
non-cutpoint  with $d(u)=3$, then there is exactly one neighbor $v\in N(u)$ such that $N[v]\backslash N[u] \neq \emptyset$.
\end{lemma}
\begin{proof}
Recall that
\begin{equation}\label{noncut3case1}
\begin{aligned}
|M_G|&\leq  |M_{G-u}|+\sum_{p\in N(u)}|M_{G-N[u]\cup N[p]}|.
\end{aligned}
\end{equation}
%In order to obtain a better upper bound on $|M_G|$, we first give the following claim.
%{\bf Claim 2.} {\it  $|M_G|\leq f(n)$.}\\
By Lemmas \ref{botool1} and \ref{d43tool}, it holds that for any $p\in N(u)$ and
any component $C$ on  $n_c$  vertices of subgraph $G-N[u]\cup N[p]$, $|M_C|\leq f(n_c)\leq 0.85g(n_c).$
Denote by $N(u)=\{v_1,v_2,v_3\}$ and $X=\{v_i\in N(u)~|~|N(v_i)\backslash N[u]|\geq1\}$.
We claim that if $|X|\geq 2$, then $|M_G|\leq f(n)$, contradicting the choice of $G$.

\noindent {\bf Case 1.} {\it  $|X|=3$.}

Similar to the discussion of Case 1 in Lemma \ref{zy2}.
When $n\geq 24$, Ineq. (\ref{noncut3case1}) yields
$|M_G|\leq  f(n-1)+3\cdot 0.85^{2}g(n-5)<0.9745\cdot 10^\frac{n-1}{5}+f_1(n-1)<f(n)$. For
$15\leq n\leq 23,$  $|M_G|\leq f(n)$ by a direct verification. It follows that $|X|\neq 3$.

\noindent {\bf Case 2.} {\it  $|X|=2$ and let $X=\{v_1,v_2\}$.}

Since $u$ is a non-cutpoint and $N[v_3]\subseteq N[u],$  $v_3$ is also a non-cutpoint and thus $d_G(v_3)=3$ by Corollary   \ref{d54noncut}. If   $G-N[u]$  is connected, then it is easy to see that
$v_1, v_2$ are both non-cutpoints with $d_G(v_1)=d_G(v_2)=3$. Now the place of $v_1$ is equivalent to $u.$   However, any neighbor $p\in N(v_1)$ satisfies $|N(p)\backslash N[v_1]|\geq 1$, which contradicts   Case $1$. So we find that $G-N[u]$  is disconnected. If there exists some $v_i\in X$ with $|N(v_i)\backslash N[u]|\geq 2$, then when $n\geq 25$, Ineq. (\ref{noncut3case1}) yields
\begin{equation*}
\begin{aligned}
|M_G|&\leq  f(n-1)+ 0.85^{2}g(n-4)+0.85^{2}g(n-6)+0.85^{2}g(n-5)\\
&<0.9992\cdot 10^\frac{n-1}{5}+f_1(n-1)<f(n).
\end{aligned}
\end{equation*} For
$15\leq n\leq 24,$  $|M_G|\leq f(n)$ by a direct verification.
It follows that $|N(v_1)\backslash N[u]|=|N(v_2)\backslash N[u]| =1$ and they are both cutpoints as $G-N[u]$  is disconnected.
If $v_1v_2\in E(G)$, then by Lemmas \ref{botool1} and \ref{d43tool}, we have
\begin{align*}
  |M_G| & \leq |M_G(\{v_1\},\{v_2\})|+| M_G(\{v_2\},\{v_1\})| +|M_G(\emptyset,\{v_1, v_2\})| +|M_G(\{v_1, v_2\}, \emptyset)|\\
   & \leq 2\cdot  \left( 2\cdot 0.85^2g(n-5)+ g(n-6)\right)+ 0.85^2g(n-6)+ 0.85^2g(n-4)\\
   &<0.9118\cdot 10^\frac{n-1}{5} < f(n),
\end{align*}which contradicts the choice of $G$.
If $v_1v_2\notin E(G)$, then by Lemmas \ref{botool1} and \ref{d43tool}, we have
\begin{align*}
  |M_G| & \leq |M_G(\{v_1\},\{v_2\})|+| M_G(\{v_2\},\{v_1\})| +|M_G(\emptyset,\{v_1, v_2\})| +|M_G(\{v_1, v_2\}, \emptyset)|\\
   & \leq 2\cdot  \left( 2\cdot 0.85^2g(n-5)+ g(n-6)\right)+ g(n-8)+  0.85^2g(n-4)\\
   &<0.8794\cdot 10^\frac{n-1}{5} < f(n),
\end{align*}
which contradicts the choice of $G$. Then $|X|\neq 2$.
This completes the proof of Lemma \ref{zy3}.
\end{proof}

%\begin{lemma} (Yuan) \label{bo4}
%Let $u$ be  a vertex of maximum degree of  $G$. Then $u$ is a cutpoint.
%\end{lemma}
%\noindent {\bf Proof.} Denote by $d(u)=\Delta(G)=:x.$ If $d(u)\geq 6,$ then the result follows from Lemma  \ref{bolem1}. It is easy to see that  $d(u)=\Delta(G)\notin \{1,2\}$, otherwise $G$ is a path  or  a cycle.
%
%We only need to consider the cases where $3\leq d(u)\leq 5$ and $u$ is not a  cutpoint.
%If so, Lemma by Zheng implies that $u$ is contained in a pendant block $K_{x+1}$ and that there exists a neighbor
%$p\in N(u)$ such that $N[p]\backslash N[u]\neq \emptyset$. Thus, $d(p)>d(u)=\Delta(G),$ a contradiction. The result follows.

\begin{corollary}  \label{d3noncut}
Let $G$ be a minimal counterexample.  Then $G$ does not have non-cutpoints.
\end{corollary}
\begin{proof}
If not, let $u$ be  a non-cutpoint of  $G$. Then $d(u)=3$ by Corollary  \ref{d54noncut}.
Lemma \ref{zy3} shows that   there is exactly one neighbor $v\in N(u)$ such that $N[v]\backslash N[u] \neq \emptyset$. It is clear that $v$ is a cutpoint of $G$, and thus the subgraph $G[N[u]]= K_4$ by Lemma \ref{gcut5}, which contradicts   Lemma \ref{bopen05}.
\end{proof}

%\section{Proof of  Theorem \ref{main1} for  $n\geq 14$\label{proofmain1}}

\begin{proof}[\bf Proof of Theorem~\ref{main1}.]
Let $u$ be  a non-cutpoint of  $G$.  Then $u$ does not exist by  Corollary  \ref{d3noncut}.
This implies that   $G$ does not exist and
the result follows.

It remains to show that the bound in Theorem \ref{main1} is the best possible.
For  $n\equiv 1~(\bmod~5)$ and $15\leq n\le 30$, let $G=K_6\ast \frac{n-6}{5}K_5$;
for $n\equiv 1~(\bmod~5)$ and $n\geq 31$, let $G=K_1\ast \frac{n-1}{5}K_5$.
By Corollary \ref{bocor2}, $|M_G|=f(n)$.
This completes the proof of Theorem \ref{main1}.
\end{proof}

\section{Concluding remarks}
In this paper, we establish upper bounds on the number of maximal and maximum induced matchings in connected graphs of order $n$.
We identify the extremal graphs for maximal induced matchings for $1\le n\le 13$ and $n\ge 14$ with $n \equiv 1 ~({\rm mod}~5)$ and maximum induced matchings for $1\le n\le 8$ and $n\ge 31$ with $n \equiv 1 ~({\rm mod}~5)$. Finally, we conjecture the exact value of the maximum number of maximal induced matchings in a connected graph of order $n$ as follows.
\begin{conjecture}
Let $G$ be a connected graph of order $n$. Then $|M_G|\leq I(n)$
and the equality holds if and only if $G= F_n$, where
  \begin{align*}
I(n):=
\begin{cases}
{n\choose2} &~ {\rm if}~ 1\le n\le 8; \\
{\binom{\lfloor\frac{n}{2}\rfloor}{2}\cdot\binom{\lceil\frac{n}{2}\rceil}{2}-(\lfloor\frac{n}{2}\rfloor-1)\cdot(\lceil\frac{n}{2}\rceil-1)+1} &~ {\rm if}~ 9\le n\le 13; \\
36\cdot10^{\frac{n-9}{5}}+\frac{3n+363}{5}\cdot6^{\frac{n-14}{5}} &~ {\rm if}~ n=4~ ({\rm mod}~5) ~ {\rm and} ~ 14\le n\le 44; \\
36\cdot10^{\frac{n-9}{5}}+\frac{9n+99}{5}\cdot6^{\frac{n-14}{5}} &~ {\rm if}~ n=4~ ({\rm mod}~5)~ {\rm and}~  n> 44;\\
6\cdot10^{\frac{n-5}{5}}+\frac{n+115}{5}\cdot6^{\frac{n-10}{5}} & ~ {\rm if}~n=0~ ({\rm mod}~5)~ {\rm and} ~  15\le n\le 50; \\
6\cdot10^{\frac{n-5}{5}}+\frac{3n+15}{5}\cdot6^{\frac{n-10}{5}} &~ {\rm if}~ n=0~ ({\rm mod}~5)~ {\rm and} ~  n> 50; \\
10^{\frac{n-1}{5}}+\frac{n+144}{30}\cdot6^{\frac{n-6}{5}} &~ {\rm if}~ n=1~ ({\rm mod}~5) ~ {\rm and} ~ 16\le n\le 30; \\
10^{\frac{n-1}{5}}+\frac{n-1}{5}\cdot6^{\frac{n-6}{5}} &~ {\rm if}~ n=1~ ({\rm mod}~5) ~ {\rm and} ~  n> 30; \\
15\cdot10^{\frac{n-7}{5}}+\frac{n+141}{3}\cdot6^{\frac{n-12}{5}} &~ {\rm if}~ n=2~ ({\rm mod}~5) ~ {\rm and} ~ 17\le n\le 32; \\
15\cdot10^{\frac{n-7}{5}}+(2n-8)\cdot6^{\frac{n-12}{5}} &~ {\rm if}~ n=2~ ({\rm mod}~5) ~ {\rm and} ~  n> 32;\\
225\cdot10^{\frac{n-13}{5}}+\frac{5n+690}{9}\cdot6^{\frac{n-13}{5}} &~ {\rm if}~ n=3~ ({\rm mod}~5) ~{\rm and}~  18 \le n\le 33; \\
225\cdot10^{\frac{n-13}{5}}+\frac{10n-70}{3}\cdot6^{\frac{n-13}{5}} & ~ {\rm if}~ n=3~ ({\rm mod}~5)~ {\rm and} ~  n> 33. \\
\end{cases}
\end{align*}
and
\begin{align*}
F_n:=
\begin{cases}
K_n &~ {\rm if}~ 1\le n\le 7; \\
K_8~{\rm or}~ K_4\ast K_4 &~ {\rm if}~ n=8; \\
K_{\lfloor \frac{n}{2}\rfloor}\ast K_{\lceil \frac{n}{2} \rceil}
 &~ {\rm if}~ 9\le n\le 13; \\
K_5\ast(\frac{n-9}{5}K_5\cup K_4) &~ {\rm if}~ n=4~ ({\rm mod}~5) ~ {\rm and} ~ 14\le n\le 44; \\
K_1\ast(\frac{n-9}{5}K_5\cup 2K_4) &~ {\rm if}~ n=4~ ({\rm mod}~5)~ {\rm and}~  n> 44;\\
K_5\ast\frac{n-5}{5}K_5 & ~ {\rm if}~n=0~ ({\rm mod}~5)~ {\rm and} ~  15\le n\le 50; \\
K_1\ast(\frac{n-5}{5}K_5\cup K_4) &~ {\rm if}~ n=0~ ({\rm mod}~5)~ {\rm and} ~  n> 50; \\
K_6\ast\frac{n-6}{5}K_5 &~ {\rm if}~ n=1~ ({\rm mod}~5) ~ {\rm and} ~ 16\le n\le 30; \\
K_1\ast\frac{n-1}{5}K_5 &~ {\rm if}~ n=1~ ({\rm mod}~5) ~ {\rm and} ~  n> 30; \\
K_6\ast(\frac{n-12}{5}K_5\cup K_6) &~ {\rm if}~ n=2~ ({\rm mod}~5) ~ {\rm and} ~ 17\le n\le 32; \\
K_1\ast(\frac{n-7}{5}K_5\cup K_6) &~ {\rm if}~ n=2~ ({\rm mod}~5) ~ {\rm and} ~  n> 32;\\
K_6\ast(\frac{n-18}{5}K_5\cup 2K_6) &~ {\rm if}~ n=3~ ({\rm mod}~5) ~{\rm and}~  18 \le n\le 33; \\
K_1\ast(\frac{n-13}{5}K_5\cup 2K_6) & ~ {\rm if}~ n=3~ ({\rm mod}~5)~ {\rm and} ~  n> 33. \\
\end{cases}
\end{align*}

\end{conjecture}

\end{document}